\documentclass[final,3p,times,authoryear]{elsarticle}
 \pdfoutput=1
\usepackage{amsthm}
\usepackage{amssymb}
\usepackage{amsmath}
\usepackage{bm}
\usepackage{framed}
\renewcommand{\P}{\mathbb{P}}

\DeclareMathOperator{\eqd}{\stackrel{\mathrm{d}}{=}}

\DeclareMathOperator{\E}{\mathbb{E}}
\DeclareMathOperator{\Z}{\mathbb{Z}}
\DeclareMathOperator{\R}{\mathbb{R}}

\DeclareMathOperator{\B}{\mathcal{B}}

\DeclareMathOperator{\vto}{\stackrel{\mathrm{v}}{\to}}
\DeclareMathOperator{\wto}{\stackrel{\mathrm{w}}{\to}}
\DeclareMathOperator{\iidsim}{\stackrel{\mathrm{\mathrm{iid}}}{\sim}}

\DeclareMathOperator{\Var}{\mathrm{Var}}
\DeclareMathOperator{\Cov}{\mathrm{Cov}}

\renewcommand{\[}{\left[}

\renewcommand{\]}{\right]}
\renewcommand{\(}{\left(}
\renewcommand{\)}{\right)}

\newcommand{\e}{\mathrm{e}}
\DeclareMathOperator{\N}{\mathbb{N}}

\newtheorem{theorem}{Theorem}
\newtheorem{lemma}{Lemma}
\newtheorem{proposition}{Proposition}
\newdefinition{definition}{Definition}
\newdefinition{notation}{Notation}
\newdefinition{corollary}{Corollary}
\newdefinition{remark}{Remark}

\journal{Stochastic Processes and their Applications}
\begin{document}
\begin{frontmatter}
 \title{Hawkes and INAR($\infty$) Processes}
 \author{Matthias Kirchner}
 \ead{matthias.kirchner@math.ethz.ch}
\address{Department of Mathematics, ETH Zurich, Raemistrasse 101, 8092 Zurich, Switzerland}
\date{\today}

\begin{abstract}
In this paper, we discuss integer-valued autoregressive time series (INAR), Hawkes point processes, and their interrelationship. Besides presenting structural analogies, we derive a convergence theorem. More specifically, we generalize the well-known INAR($p$), $p\in\N$, time series model to a corresponding model of infinite order: the INAR($\infty$) model. We establish existence, uniqueness, finiteness of moments, and give formulas for the autocovariance function as well as for the joint moment-generating function. Furthermore, we derive an AR($\infty$), an MA($\infty$), and a branching-process representation for the model. We compare Hawkes process properties with their INAR($\infty$) counterparts. Given a Hawkes process $N$, in the main theorem of the paper we construct an INAR($\infty$)-based family of point processes and prove its convergence to $N$. This connection between INAR and Hawkes models will be relevant in applications.
\end{abstract}
\begin{keyword}
Hawkes process\sep integer-valued time series\sep weak convergence of point processes\sep branching process
%
\MSC 60G55 \sep 60F99 \sep 37M10 
\end{keyword}
\end{frontmatter}
\section*{Introduction}
\noindent In this paper, we show that Hawkes point processes are continuous-time versions of integer-valued autoregressive time series and---vice versa---that integer-valued autoregressive time series are discrete-time versions of Hawkes point processes; see Theorem~\ref{weak_convergence} for the main result of the paper. To start with, we outline the history of the concepts involved.
\par Standard time series theory for sequences of real-valued data points has been developed in seminal works like \citet{whittle51} and \citet{box70}. This theory led to the natural question of time series models for count data. In the count-data context, the starting point is also a defining system of difference equations of the form ``$X_n - \sum \alpha_k X_{n-k} =\varepsilon_n+ \sum \beta_k\varepsilon_{n-k},\, n\in\Z $''. The main idea of the construction is to manipulate these equations in such a way that their solutions are integer-valued. This can be achieved by 
giving the error terms ``$(\varepsilon_n)$'' a distribution supported on $\N_0$ and by substituting all multiplications with thinning-operations.
 For the latter, the thinning notation from \citet{steutel79} turns out to be relevant. In the above spirit, autoregressive integer-valued (INAR) time series were defined and examined by \citet{mckenzie85} and \citet{alosh87}. The modern definition of the INAR model comes from \citet{li91}. \citet{latour97} generalizes the model to the multivariate case. For an exhaustive collection of properties of the INAR model; see \citet{marques05}. For a general survey of integer-valued time series whose definitions involve a thinning operation; see \citet{weiss08}. For a textbook reference; see \citet{fokianos01}.
\par The Hawkes process was introduced in \citet{hawkes71a, hawkes71b} as a model for contagious processes such as measles infections or hijackings. As a point process in continuous time, the Hawkes process allows for the modeling of intensities which depend on the past of the process itself. Its alternative name, ``selfexciting point process", stems from the fact that, given the occurrence of an event, intensity jumps upwards and then decays gradually. Theoretical cornerstones for the model are 
 \citet{hawkes74} which establishes the representation as a cluster process,
 \citet{ogata88} which covers calibration issues and propagates a recursive method for likelihood calculations,
 \citet{bremaud96} which extends the original model by generalizing the affine dependence on the past to Lipschitz dependence,
 \citet{bremaud01} which proves the existence of a specific borderline case of the model,
  \citet{liniger09} which puts the subtleties of the definition and the construction on a solid and very detailed mathematical foundation---especially for the marked multivariate case, and
 \citet{errais10} which treats an important special parametrization of the model from a Markov process theory point of view.
For a textbook reference that covers many aspects of the Hawkes process; see \citet{daley03}.
\par 
To the best of our knowledge, the close connection between INAR and Hawkes processes has not been studied before. The correspondence between the model classes becomes even more direct if one applies infinite autoregression instead of finite autoregression for the time series model. This was our main motivation for generalizing the existing INAR($p$) framework with $p<\infty$ to the case $p=\infty$. For the new INAR($\infty$) model, we give an explicit construction and show uniqueness. Then we derive three alternative descriptions of the process, namely an autoregressive, a moving-average, and a branching-process representation. Furthermore, we calculate basic quantities such as the joint moment-generating function and the autocovariance function. These are mainly presented for comparison with their Hawkes process counterparts. 
We observe that the equivalent branching-construction of INAR and Hawkes models form the core of the connection. This equivalence yields corresponding equations for generating functions, similar moment structures, and analogous stability criteria. 
Theorem~\ref{weak_convergence} establishes a convergence result. In this theorem, for a given Hawkes process $N$, we construct a specific family of INAR($\infty$) sequences $\left\{\(X^{(\Delta)}_n\)_{n\in\Z}\right\}_{\Delta>0}$. From each member of this family, we derive a point process $N^{(\Delta)}$ by setting 
 $$
 N^{(\Delta)}\big((a,b]\big):=\sum_{n:\, \Delta n\in(a,b]} X^{(\Delta)}_n,\quad a<b.
 $$
The theorem states that $N^{(\Delta)}$ converges weakly to the Hawkes process $N$ when $\Delta$ goes to zero. This result is relevant for applications of INAR and Hawkes processes. In particular, the convergence theorem yields an estimation method for the Hawkes process by estimating the more tractable approximating INAR model instead. We work out this estimation method in \citet{kirchner15a}.
Moreover, from a purely theoretical point of view, the presented line of thought is useful: the time series perspective on point processes as well as the point process perspective on (integer-valued) time series can be fertile for constructing and understanding event-data models; see Section~\ref{Discussion}.
\par The paper is organized as follows: Section 1 introduces the  INAR($\infty$) model. Section 2 presents the Hawkes process in the random counting-measure framework. Section 3 establishes the convergence theorem. Furthermore, it collects structural analogies between the two model classes. In the final section, we conclude with a discussion on the broader interpretation of the INAR--Hawkes relation. 
\section{The INAR($\infty$) model}
\noindent Throughout the paper, we consider a basic complete probability space $\(\Omega, \mathcal{F}, \P\)$ carrying all random variables involved.
\subsection{Definition and existence}
\begin{definition}\label{thinning_operator}
For an $\N_0$-valued random variable $Y$ and a constant $\alpha\geq 0$, the \emph{thinning operator} $\circ$ is defined by
$$
\alpha \circ Y:=\sum\limits_{n=1}^Y\xi^{(\alpha)}_n,
$$ where $\xi^{(\alpha)}_n\iidsim \mathrm{Pois}(\alpha)$, $n\in\N, $ independently of $Y$. We refer to $\(\xi^{(\alpha)}_n\)$ as \emph{counting sequence}.
\end{definition}
In this definition and throughout the paper, we use the convention that $\sum_{n=1}^0a_n := 0$ for any sequence $(a_n)_{n\in\N}\subset \R$. 
\begin{definition} \label{INAR}
For $\alpha_k \geq 0,\, k \in\N_0$, let $\varepsilon_n  \iidsim\operatorname{Pois}(\alpha_0) ,\, n\in\mathbb{Z}$, and $\xi^{(n,k)}_{l}\sim \mathrm{Pois}\(\alpha_k\)$, independently over $n\in\Z,k\in\N$, $l\in\N$, and also independent of $(\varepsilon_n)$. An \emph{integer-valued autoregressive time series of infinite order} (INAR($\infty$)) is a sequence of random variables $(X_n)_{n\in\mathbb{Z}}$ which is a solution to the system of stochastic difference equations
\begin{align}
 \varepsilon_n &= X_n -  \sum\limits_{k=1}^\infty \alpha_k \circ X_{n-k} \label{sloppy}\\
 &:=  X_n -  \sum\limits_{k=1}^\infty \sum\limits_{l=1}^{X_{n-k}}\xi^{(n,k)}_{l} ,\quad n\in\mathbb{Z} \label{DE}.
\end{align}
 We call $\alpha_0$ \emph{immigration parameter}, $(\varepsilon_n)$ \emph{immigration sequence}, $\alpha_k \geq 0,\, k \in\N,$ \emph{reproduction coefficients}, and $K:=\sum_{k=1}^{\infty}\alpha_k$ \emph{reproduction mean}.
\end{definition}
\noindent In most situations, it is enough to use the thinning notation from~\eqref{sloppy} in Definition~\ref{INAR} without explicitely writing out the counting sequences as in~\eqref{DE}---keeping in mind that each ``$\circ$" operates independently over $k\in\N$ and $n\in\Z$. Clearly, Definitions~\ref{thinning_operator} and~\ref{INAR} depend on the choice of the distribution of the counting sequences. A more obvious option would have been sequences of Bernoulli variables. This has in fact been the choice in the cited INAR($p$) literature. The Poisson choice for the counting sequences. however, again yields a Poisson distribution for ``$X_n|X_{n-1},X_{n-2},\dots$''. This in turn leads to formulas that are simpler and that can be compared with their Hawkes counterparts more directly. We will address this issue in Sections~\ref{The Choice of The Counting Sequence Distribution} and~\ref{Discussion}. For the following existence and uniqueness result, any $\N_0$-valued distribution with finite first moments would do for the counting sequences. Throughout our paper, ``stationary'' is understood as ``strictly stationary''.

\begin{theorem}\label{existence}
Let $\alpha_k\geq 0,\, k \in\N_0$, with reproduction mean $K:=\sum_{k=1}^\infty \alpha_k<1$. Then  \eqref{DE} has an almost surely unique stationary solution $\(X_n\)_{n\in\Z}$, where $X_n\in\N_0,\, n\in\Z,$ and  
$
\E X_n  \equiv {\alpha_0}/({1 - K}),\, n\in\N.
$
\end{theorem}
\noindent Before giving the proof, we highlight the branching nature of the solution to \eqref{DE}. As we will see, one can interpret this solution as a model for the size of a population, where each individual is alive exactly during one time-step. Furthermore, each individual is either an immigrant or stems from a prior individual. This is similar to a Galton--Watson framework with immigration; see Section 5 in \citet{seneta69}. In contrast to the Galton--Watson setup, however, each INAR($\infty$) individual does not only have offspring at the next time-step but (potentially) at any future time. 
%
The proof first formalizes this structure and then establishes that the construction indeed yields a process with the desired properties. We emphasize the branching intuition of family processes consisting of generation processes by the somewhat unusual but suggestive notation for stochastic processes $(F_n)$ and  $(G_n)$.
\begin{proof}
Let $\varepsilon_i\iidsim\text{Pois}\(\alpha_0\),\,  i\in\Z$,  be the immigration terms 
from Definition~\ref{INAR}.
For the $j$-th potential immigrant at time $i\in\Z$, we define generation processes $\big(G^{(g,i,j)}_n\big),\, g\in\N_0,$ by the following recursive procedure:
\begin{align}
G_{n}^{(0,i,j)} :=&\,1_{\{n = 0\}},\quad n\in\Z,\,i\in\Z,\,j\in\N\\
 G_{n}^{(g,i,j)}:=&\,\sum_{k=1}^{n} \alpha_k\circ G_{n-k}^{(g-1,i,j)}\label{G_n_sloppy}.
 \end{align}
For all distributional properties of the construction, it will be enough to apply defining equation \eqref{G_n_sloppy} for the generations. We will consider an explicit representation of the involved thinning sequences later. Note that $ G_{n}^{(g,i,j)} = 0$ whenever $n<0$. 
A family originating from the $j$-th individual immigrant at time $i$ is the superposition of all the corresponding generation-processes:
\begin{align}
F_{n}^{(i,j)} &:= \sum\limits_{g = 0}^\infty G_n^{(g,i,j)}, \quad n\in\Z,\, i\in\Z,\, j\in\N\label{F_n1}.
\end{align}
The candidate series $\(\tilde{X}_{n}\)$ for a solution of~\eqref{DE} is the superposition of all these families---modulo an appropriate shift in the time index:
\begin{align}
\tilde{X}_{n} &:= \sum\limits_{i = -\infty}^n \sum\limits_{j=1}^{\varepsilon_i} F_{n-i}^{(i,j)},\quad n\in\Z \label{tildeX}.
\end{align}
Note that only family and generation processes indexed by $(i,j)$ with $j\in \{1,\dots,\varepsilon_i\}$ come into play. 
Also note that all infinite series involved in the construction above are well-defined because their partial sums are nondecreasing.
As a first step, we remind ourselves that $K = \sum_{k=1}^\infty \alpha_k$  and establish 
 \begin{equation}
  \E\sum_{n=0}^\infty G_n^{(g,i,j)} = K^g \label{GSum},\quad  g\in\N_0,\,i\in\Z,\, j\in\N,
  \end{equation} by induction:
  for $g=0$,~\eqref{GSum} is correct because we have 
  $
  \E \sum_{n=0}^\infty G_{n}^{(0,i,j)} =  \sum_{n=0}^\infty 1_{\{n = 0\}} =1 = K^0.
  $
For $g>0$,  
 one can show that
$
  \E\sum_{n=0}^\infty G_n^{(g,i,j)} 
= K \cdot \E \sum_{n=0}^\infty G_{n}^{(g-1,i,j)}
$ 
and $\eqref{GSum}$ follows. 
By \eqref{GSum},
$\sum_{n=0}^\infty F_{n}^{(i,j)}$ has expectation $1/(1-K)$. In particular, this expectation is almost surely bounded. We conclude that $F_{n}^{(i,j)}\in\N_0,\,n\in\Z$, almost surely. For $\tilde{X}_{n},\, n\in\Z,$ we find that
\begin{align*}
&E \tilde{X}_{n}\\
& \stackrel{\mathrm{MC}}{=}\sum\limits_{i = -\infty}^n\E \sum\limits_{j=1}^{\varepsilon_i} F_{n-i}^{(i,j)}
\stackrel{\mathrm{WI}}{=}\alpha_0\sum\limits_{i = -\infty}^n\E F_{n-i}^{(i,1)}
\stackrel{\mathrm{MC}}{=}\alpha_0\sum\limits_{i = -\infty}^n \sum\limits_{g = 0}^\infty\E G_{n-i}^{(g,i,j)}
\stackrel{\mathrm{MC}}{=}\alpha_0 \sum\limits_{g = 0}^\infty\E \sum\limits_{i = -\infty}^nG_{n-i}^{(g,i,1)}
\stackrel{\eqref{GSum}}{=}\alpha_0 \sum\limits_{g = 0}^\infty K^g = \frac{\alpha_0}{1-K}.
\end{align*}
Consequently, $\tilde{X}_{n}$ is almost surely bounded and therefore almost surely $\N_0$-valued.  Note that, by construction, the generations $\(G_n^{(i,j)}\)$ and therefore the families $\(F_n^{(i,j)}\)$ are independently and identically distributed time series over $i\in\Z$ and $j\in\N$. Stationarity of $\(\tilde{X}_{n}\)$ then follows from  the i.i.d.\ property of the immigration sequence $\(\varepsilon_i\)$.
\par To show that our candidate sequence $\big(\tilde{X}_n\big)$ indeed solves \eqref{DE}, we have to work with an explicit representation for the thinnings involved in the~\eqref{G_n_sloppy}. To that aim, let 
\begin{align}
\xi_{g,i,j,m}^{(n,k)}\sim\mathrm{Pois}(\alpha_k),\quad\text{independently over}\quad i,n\in \Z\;\text{and}\;k, j,g,m\in\N.
\end{align}
In our branching terminology, $\xi_{g,i,j,m}^{(n,k)}$ denotes the number of offspring individuals at time $n$ whose parent lived at time $n-k$. 
This parent belongs to the $(g-1)$-th generation of family $(i,j)$. Furthermore, this parent is the $m$-th such individual. We repeat the defining recursion for the generation processes from above---this time we represent the involved counting sequences explicitly:
\begin{align}
G_{n}^{(0,i,j)} :=&\,1_{\{n = 0\}},\quad n\in\Z,\,i\in\Z,\,j\in\N\\
 G_{n}^{(g,i,j)}:=&\,\sum_{k=1}^{n} \alpha_k\circ G_{n-k}^{(g-1,i,j)}\label{G_n_sloppy_2}\\
 :=&\, \sum_{k=1}^{n} \sum_{m=1}^{G_{n-k}^{(g-1,i,j)}} \xi_{g,i,j,m}^{(i + n,k)},\quad n\in\Z,\,i\in\Z,\,\,j\in\N,\, g\in\N\label{G_n_2}.
 \end{align}
It is obvious that~\eqref{G_n_2} justifies the distributional assumptions on \eqref{G_n_sloppy_2} (i.e., \eqref{G_n_sloppy}), used in the first part of this proof. For any $n\in\Z$, we find
\begin{align}
\tilde{X}_n &= \sum\limits_{i=-\infty}^n\sum\limits_{j=1}^{\varepsilon_i}F_{n-i}^{(i,j)}
=\sum\limits_{i=-\infty}^{n-1}\sum\limits_{j=1}^{\varepsilon_i}F_{n-i}^{(i,j)} + \sum\limits_{j=1}^{\varepsilon_n}F_{0}^{(n,j)}=\sum\limits_{i=-\infty}^{n-1}\sum\limits_{j=1}^{\varepsilon_i} \sum\limits_{g = 0}^\infty G_{n-i}^{(g,i,j)}+\varepsilon_n. \label{third}
\end{align}
Note that the third summation really starts in $g=1$ because $G_{n-i}^{(0,i,j)} = 1_{\{n-i=0\}}=0$ whenever $i\leq n-1$. For the triple sum in \eqref{third}, we obtain
\begin{align}
\sum\limits_{i=-\infty}^{n-1}\sum\limits_{j=1}^{\varepsilon_i} \sum\limits_{g = 1}^\infty G_{n-i}^{(g,i,j)}&=\sum\limits_{i=-\infty}^{n-1}\sum\limits_{j=1}^{\varepsilon_i} \sum\limits_{g = 1}^\infty  \sum_{k=1}^{n-i} \sum_{m=1}^{G_{n-i-k}^{(g-1,i,j)}} \xi_{ g,i,j,m}^{(i+n-i,k)}
\nonumber\\
&= \sum\limits_{k=1}^\infty\(\sum\limits_{i=-\infty}^{n-k}\sum\limits_{j=1}^{\varepsilon_i} \sum\limits_{g = 1}^\infty \sum\limits_{m=1}^{G_{n-k-i}^{(g-1,i,j)}}  \xi_{g,i,j,m}^{(n,k)}\).\label{2nd_line}
\end{align}
For~\eqref{2nd_line}, we use the fact that it is irrelevant whether we let $k$ run over $\{1,2,\dots, n-i\}$ or over $\N$ because $G_{n-i-k}^{(g-1,i,j)}=0$ whenever $k> n-i$; by the same argument, we may let $i$ run up to $n-k$ only. For fixed $n\in\Z$ and $k\in\N$, the term in the bracket is a sum of i.i.d. Pois($\alpha_k$) random variables $\xi_{n,j,g,m}^{(n+i,k)}$ over the (stochastic) index set 
$$
I^{(n,k)} := \bigg\{(g,i,j,m)\in\Z^4:\,1\leq g,\,   i\leq n-k,\, 1\leq j\leq \varepsilon_i,\,1\leq m\leq G_{n-k-i}^{(g-1,i,j)}\bigg\}.
$$
For the size of $I^{(n,k)}$, we obtain
\begin{align}
\left|I^{(n,k)}\right|=\sum\limits_{i=-\infty}^{n-k}\sum\limits_{j=1}^{\varepsilon_i} \sum\limits_{g = 1}^\infty G_{n-i-k}^{(g-1,i,j)}\stackrel{\eqref{F_n1}}{=}\sum\limits_{i=-\infty}^{n-k}\sum\limits_{j=1}^{\varepsilon_i} F_{n-i-k}^{(i,j)}\stackrel{\eqref{tildeX}}{=}\tilde{X}_{n-k} \quad(<\infty,\, a.s).\label{domain_size}
\end{align}
Let $(\xi_m^{(n,k)})$ be the counting sequences from Definition~\ref{INAR}.
Note that, for $n\in\Z$ and $k\in\N$,
\begin{align}
\Big(\xi_l^{(n,k)}:\, l= 1,2,\dots,{\tilde{X}_{n-k}}^{(n,k)}\Big)\quad\text{and}
\quad
\Big(\xi^{(n,k)}_{g,i,j,m}:\,(g,i,j,m)\in I^{(n,k)}\Big)\label{index_sets}
\end{align}
are equally distributed---no matter which order we choose for the second set. Also note that, for $(n,k)\neq(n',k')$, we have that $I^{(n,k)}\cap I^{(n',k')}= \emptyset$. 
Consequently, the independence properties over $n$ and $k$ necessary for the thinnings are preserved. We did indeed make correspondence \eqref{index_sets} explicit. The reordering of the counting sequences, however, is cumbersome. It involves a function that is recursively defined on multiple levels; its presentation would double the length of the whole proof---yielding hardly additional insight at that. So we chose to leave it with~\eqref{index_sets}. Continuing with~\eqref{third}, we obtain that, for $n\in\Z$,
\begin{align*}
 \tilde{X}_{n} &\stackrel{\eqref{third}}{=}\sum\limits_{i=-\infty}^{n-1}\sum\limits_{j=1}^{\varepsilon_i} \sum\limits_{g = 0}^\infty G_{n-i}^{(g,i,j)}+\varepsilon_n\stackrel{\eqref{2nd_line}}{=}\sum\limits_{k=1}^{\infty}\sum\limits_{(g,i,j,m)\in I^{(n,k)}}  \xi^{(n,k)}_{g,i,j,m} + \varepsilon_n
 = \sum\limits_{k=1}^{\infty}\sum\limits_{l=1}^{|I^{(n,k)}|} \xi^{(n,k)}_{l} + \varepsilon_n
 \stackrel{\eqref{index_sets}}{=}  \sum\limits_{k=1}^{\infty}\sum\limits_{l=1}^{\tilde{X}_{n-k}} \xi^{(n,k)}_{l} + \varepsilon_n.
\end{align*}
We conclude that $\(\tilde{X}_{n}\)$ indeed solves \eqref{DE}.
\par For uniqueness, consider two stationary solutions $X, Y$ of \eqref{DE}---naturally with respect to the same counting sequences $\(\xi^{(n,k)}_l\),\, n\in\mathbb{Z},\,k\in\N$. Then 
\begin{align}
\E\big|X_n - Y_n\big| &\stackrel{\eqref{sloppy}}{=} \E\left|\sum\limits_{k=1}^\infty \(\alpha_{k} \circ X_{n-k} - \alpha_{k}\circ Y_{n-k}\)\right| \leq  \E\sum\limits_{k=1}^\infty \big| \alpha_{k} \circ X_{n-k} - \alpha_{k}\circ Y_{n-k}\big| \stackrel{\mathrm{MC}}{=}  \sum\limits_{k=1}^\infty \E\big| \alpha_{k} \circ X_{n-k} - \alpha_{k}\circ Y_{n-k}\big|. \label{abs_diff}
\end{align}
As  $X_{n-k}$ and $Y_{n-k}$ are thinned with respect to the same counting sequence, we have that
\begin{equation}
\big| \alpha_{k} \circ X_{n-k} - \alpha_{k}\circ Y_{n-k}\big| = \left|\sum\limits_{i=1}^{X_{n-k} }\xi_{i}^{(n,k)} -  \sum\limits_{i=1}^{Y_{n-k}} \xi_{i}^{(n,k)}\right| 
\stackrel{\mathrm{d}}{=} \alpha_k\circ \big| X_{n-k} - Y_{n-k}\big|,\quad k\in\N.\label{dbn_eq}
\end{equation}
Plugging \eqref{dbn_eq} in \eqref{abs_diff}, we obtain
\begin{align*}
\E\big|X_n - Y_n\big|\,{\leq}\,  \sum\limits_{k=1}^\infty \E\Big[ \alpha_{k} \circ \big|  X_{n-k} - Y_{n-k}\big|\Big] 
&\stackrel{\text{WI}}{=}  \sum\limits_{k=1}^\infty \alpha_{k} \E \big|  X_{n-k} - Y_{n-k}\big| \stackrel{\mathrm{stat.}}{=} K \E \big|  X_{n} - Y_{n}\big| ,\quad n\in\Z.
\end{align*}
As $K<1$ by assumption and $\E|X_n - Y_n|<\infty$, we get that $\E|X_n - Y_n| = 0$ and therefore $X_n = Y_n$, $n \in \Z$, almost surely.

\end{proof}
\subsection{Alternative representations}
\noindent Surprisingly, we can explicitly represent the INAR($\infty$) model as a standard AR($\infty$) model with uncorrelated errors:

\begin{proposition}\label{AR}
Let $\alpha_k\geq 0,\,k \in\N_0,$ with $K=\sum_{k=1}^\infty \alpha_k <1$, and let $\(X_n\)$ be the corresponding INAR($\infty$) process; see Definition~\ref{INAR}. Then
\begin{equation}
u_n:=X_n - \sum\limits_{k=1}^{\infty}\alpha_kX_{n-k} - \alpha_0,\quad n\in\Z, \label{AR-DEalt}
\end{equation}
defines a stationary sequence  $\(u_n\)$  with $\E u_n\equiv 0,\, n\in\N,$ and
\begin{align}\label{WN_cov}
\E\[u_n u_{n'}\] =
\begin{cases}
0,& \quad n\neq n'\text{,}\\
\frac{\alpha_0}{1-K}, &\quad n=n'.
\end{cases}
\end{align}
Furthermore, we have that
\begin{equation}
\(X_n-\mu_X\) -  \sum\limits_{k=1}^\infty \alpha_k \( X_{n-k}-\mu_X\) =  u_n,\quad n\in\mathbb{Z} \label{AR-DE},
\end{equation}
where $\mu_X:=\E X_0 = \alpha_0/(1-K)$. In other words, $\(u_n\)$ is a (dependent) white-noise sequence and the time series $\(X_n-\mu_X\)_{n\in\Z}$ can be described in terms of a solution to an ordinary AR($\infty$) system of difference equations.
\end{proposition}

\begin{proof}
The sequence values $u_n, \,n\in\Z,$ are well-defined because the partial sums of $\sum_{k=1}^\infty \alpha_{k} X_{n-k}$ are nondecreasing and their expectations have a finite limit. Stationarity of $\(u_n\)$ follows from stationarity of $\(X_n\)$. For the expectation, we find
$$
\E u_n = \E\[X_n -\alpha_0 - \sum_{k=1}^\infty \alpha_{n-k} X_{n-k}\] \stackrel{\mathrm{MC}}{=} \mu_X -\alpha_0- K \mu_X = 0.
$$ 
For the autocovariances of the errors, first of all note that the errors $u_n$ are uncorrelated with any previous value $X_{n'},\, n' <n,$ of the INAR($\infty$) sequence.
So that, for $n'<n$ (and then, by symmetry, for $n'\neq n$),
\begin{align*}
\E\[u_n u_{n'}\]&= \E\[u_n \(X_{n'}  - \alpha_0 - \sum\limits_{k=1}^\infty \alpha_k X_{n'-k} \)\]
\stackrel{\mathrm{MC}}{=}\E\[u_n X_{n'}\]  - \E\[u_n \alpha_0\]  - \sum\limits_{k=1}^\infty \alpha_k \E\[u_nX_{n'-k}\]{=}  - \E\[u_n\] \alpha_0 = 0.
\end{align*}
Since $\E u_n=0$ and $\E\[u_nX_{n-k}\] = 0$, for $k\in\N$, we obtain 
\begin{align*}
&\E\[u_n^2\] \\&= \E\[u_n\(X_n - \alpha_0 -\sum\limits_{k=1}^\infty\alpha_kX_{n-k}\)\]\\
& = \E\[u_nX_n\] \\
&= \E\[u_n\(\varepsilon_n+  \sum\limits_{k=1}^\infty \alpha_k\circ X_{n-k}\)\]\\
& =  \E\[\(X_n - \alpha_0 - \sum\limits_{k=1}^\infty\alpha_k X_{n-k}\)\(\varepsilon_n+   \sum\limits_{k=1}^\infty \alpha_k\circ X_{n-k}\)\]\\
& =   \E\[\(\varepsilon_n +   \sum\limits_{k=1}^\infty \alpha_k\circ X_{n-k} - \alpha_0 - \sum\limits_{k=1}^\infty\alpha_k X_{n-k}\)\(\varepsilon_n+   \sum\limits_{k=1}^\infty \alpha_k\circ X_{n-k}\)\]\\
&= \E\Bigg[\varepsilon_n^2
+\varepsilon_n \sum\limits_{k=1}^\infty \alpha_k\circ X_{n-k}
 +  \varepsilon_n\sum\limits_{k=1}^\infty \alpha_k\circ X_{n-k}
+ \(\sum\limits_{k=1}^\infty \alpha_k\circ X_{n-k}\)^2 \\
& \phantom{ \E[\varepsilon_n}
 - \alpha_0 \varepsilon_n   
 - \alpha_0  \sum\limits_{k=1}^\infty \alpha_k\circ X_{n-k}
  - \varepsilon_n\sum\limits_{k=1}^\infty\alpha_k X_{n-k}
    - \(\sum\limits_{k=1}^\infty\alpha_k X_{n-k}\)\(\sum\limits_{k=1}^\infty \alpha_k\circ X_{n-k}\)   
\Bigg].
\end{align*}
Using independence of $\varepsilon_n$ from the past of the process and using $\E\varepsilon_n= \Var(\varepsilon_n) = \alpha_0$, we get that 
$$
\E\[u_n^2\] = \alpha_0 + \E\[\(\sum\limits_{k=1}^\infty \alpha_k\circ X_{n-k}\)^2 - \sum\limits_{k=1}^\infty \alpha_kX_{n-k}\(\sum\limits_{k=1}^\infty \alpha_k\circ X_{n-k}\)\].
$$
For the expectation part in the latter equation, we condition the difference in the expectation on $\sigma\(X_{n-k},\, k\in\N\)$, the past of the process. Then the thinnings become the only source of randomness. Note that the counting sequences of the involved thinnings  only feed into $X_n$ and that they are independent of the past values $X_{n-1}, X_{n-2}, \dots$; see Definition~\ref{INAR}. From this independence, we obtain
\begin{align}
& \E\[\(\sum\limits_{k=1}^\infty \alpha_k\circ X_{n-k}\)^2 - \sum\limits_{k=1}^\infty \alpha_kX_{n-k}\(\sum\limits_{k=1}^\infty \alpha_k\circ X_{n-k}\)\Bigg|\sigma\big(X_{n-k},\, k\in\N\big)\]\nonumber\\
&\hspace{6cm}= \E\[\(\sum\limits_{k=1}^\infty \alpha_k\circ x_{n-k}\)^2 - \sum\limits_{k=1}^\infty \alpha_kx_{n-k}\(\sum\limits_{k=1}^\infty \alpha_k\circ x_{n-k}\)\]_{x_{n-k} = X_{n-k},\, k\in\N}
\nonumber
\end{align}
for the conditional expectations. So in a first step, we take the unconditional expectation with respect to the thinnings only and treat the $X_{n-k}$ variables as deterministic. Observe that, for deterministic $\(x_{n-k}\)_{k\in\N}\subset \N_0$ such that $\sum_{k=1}^\infty\alpha_kx_{n-k}<\infty$, we have that

\begin{align*}
\E\[  \(\sum\limits_{k=1}^\infty \alpha_k\circ x_{n-k}\)^2 - \sum\limits_{k=1}^\infty \alpha_kx_{n-k}\(\sum\limits_{k=1}^\infty \alpha_k\circ x_{n-k}\)\] 
{=}\E\[\(\sum\limits_{k=1}^\infty \alpha_k\circ x_{n-k}\)^2\] - \E\[\sum\limits_{k=1}^\infty \alpha_k\circ x_{n-k}\]^2 
= \Var\( \sum\limits_{k=1}^\infty \alpha_k\circ x_{n-k}\).
\end{align*}
 One can show by standard arguments that the variance of these series is the limit of the  variances of the partial sums, hence
\begin{align}
\Var\( \sum_{k=1}^\infty \alpha_k\circ x_{n-k}\) =\sum_{k=1}^\infty \Var( \alpha_k\circ x_{n-k}) = \sum_{k=1}^\infty \alpha_k x_{n-k} \quad(<\infty).\label{var}
\end{align}
For the last equality, we use that, for $x\in\N_0$ and $\alpha\geq 0$, $\alpha \circ x$ is the sum of $x$ independent random variables with distribution $\mathrm{Pois}(\alpha)$; see Definition~\ref{thinning_operator}. As $\sum_{k=1}^\infty \alpha_k X_{n-k}$ is almost surely finite, we may average \eqref{var} over the possible values of $X_{n-k}$ and conclude
\begin{align*}
\E\[u_n^2\]  =\alpha_0 + \E\[\Var\( \sum_{k=1}^\infty \alpha_k\circ x_{n-k}\)_{(x_n) = (X_n)}\]  =  \alpha_0 +  \E\sum\limits_{k=1}^\infty \alpha_k X_{n-k}\stackrel{\text{MC}}{=}\alpha_0 + K\mu_X = \frac{\alpha_0}{1-K}.
\end{align*}
This establishes~\eqref{WN_cov}. Equation \eqref{AR-DE} is a simple algebraic transformation of \eqref{AR-DEalt}.
\end{proof}
\noindent The third description of the INAR($\infty$) model is the representation as a standard MA($\infty$) time series. It will be most helpful for establishing the second-order properties of the process.
\begin{proposition} \label{MA}
Let $\alpha_k\geq 0,\, k\in\N_0,$ with $\sum_{k=1}^\infty\alpha_k < 1$. Then the corresponding INAR($\infty$) process from Definition~\ref{INAR} is a solution to the family of equations
\begin{equation}
X_n- \mu_X = \sum\limits_{k=0}^\infty \beta_k u_{n-k},\quad n\in\Z, \label{MA-DE}
\end{equation}
where $\(u_n\)$ is the white-noise sequence from Proposition~\ref{AR}, $\mu_X: = \E X_0 = \alpha_0/(1-K)$, $\beta_0:= 1$, and $\beta_k:= \sum_{i=1}^k \alpha_i \beta_{k-i},\; k\in\N_0$. Furthermore, $\beta_k\geq 0$, $k\in\N_0,$ and $\sum_{k=0}^\infty \beta_k = 1/(1-K)<\infty$.
\end{proposition}
\noindent Before giving the proof, we note that the moving-average coefficients $(\beta_k)_{k\in\N_0}$ defined above correspond to $\(\E F^{(i,j)}_k\)_{k\in\N_0}$, the expected values of a single family-process after $k = 0, 1, 2, \dots $ time steps; see \eqref{F_n1}. Also note that Proposition~\ref{MA} above is not a mere corollary of standard results like Theorem 3.1.1 of \citet{brockwell91}, stating that ARMA($p,q$) processes ($p,q<\infty$) are MA($\infty$) processes. The argumentation in our case, i.e., $p=\infty$, has to be more subtle: we have to prove that the (in general infinitely many) zeros of the involved power series can be bounded away from the unit circle.

\begin{proof}
Let $B$ be the backward shift operator defined by $B^kx_n:=x_{n-k},\, k\in \Z$, for any sequence $(x_n)_{n\in\Z}$.
Consider the power series $\phi\(z\):=1-\sum_{k=1}^\infty\alpha z^k$. With these notations, we may rewrite \eqref{AR-DE} as
$
\phi(B) \(X_n - \mu_X\) = u_n,\, n\in\Z,
$
where $\(u_n\)$ is the white-noise sequence from Proposition~\ref{AR}. 
 Note that $0 < |\phi(z)| < 2$ for $|z|\leq 1$. Indeed, for $|z|\leq 1$,
$$
|\phi(z)|=\left|1-\sum_{k=1}^\infty\alpha_k z^k\right| \geq 1-\left|\sum_{k=1}^\infty\alpha_k z^k\right| \geq 1-\sum_{k=1}^\infty\alpha_k |z|^k \geq 1-K > 0.
$$
and
$$
|\phi(z)|= \left|1-\sum_{k=1}^\infty\alpha_k z^k\right| \leq 1+ \left|\sum_{k=1}^\infty\alpha_k z^k\right| \leq 1+\sum_{k=1}^\infty\alpha_k |z|^k \leq 1+K < 2.
$$
As $\phi(z)\neq 0$ for $|z| \leq 1$, we may define the function
$
\psi(z):=1/\phi(z),\, |z| \leq 1.
$
The original function $\phi$ is analytic on $|z| < 1$, so $\psi$ is also analytic on the open unit-disc and, consequently, we have a power series representation
$\psi(z) = \sum_{k=0}^\infty\beta_k z^k,\, |z| < 1$.  As $1= \psi(z)\phi(z)$ by definition, if follows that, for $|z| < 1$, 
\begin{equation}
1=\sum_{k=0}^\infty\beta_k z^k\(1 - \sum_{l=1}^\infty\alpha_l z^l\) = \sum_{k=0}^\infty\(\beta_k -\sum\limits_{j=1}^k \alpha_j\beta_{k-j}\)z^k. \label{beta}
\end{equation}
Comparing coefficients in~\eqref{beta}, one obtains the recursion 
$$
\beta_0 = 1\quad\text{ and}\quad \beta_k =  \sum_{j=1}^k \alpha_j\beta_{k-j},\quad k\in\N. 
$$
We note that $\beta_{k}\geq 0$ because  $\alpha_k\geq 0$. Formally, we can write
\begin{equation}
X_n - \mu_X = \psi(B) u_n,\quad n\in\Z.\label{formal}
\end{equation}
For the well-definedness of the right-hand side of this equation, it suffices to show that $\sum_{k=0}^\infty|\beta_k| < \infty$; see Proposition 3.1.2 in \citet{brockwell91}. To that aim we apply Lemma IIc.\ of \citet{wiener32}.
Let $
\tilde{\phi}(\theta) = 1 - \sum_{k=1}^\infty \alpha_ke^{ik\pi \theta},\, \theta\in (-\pi, \pi].
$
The lemma states that if  $ \sum_{k=1}^\infty |\alpha_k|<1$, then the Fourier series of
 the function $1/\tilde{\phi}(\theta)$ is absolutely convergent. By the same calculation as in \eqref{beta}, we find that the Fourier-coefficients of  $1/\tilde{\phi}(\theta)$ are exactly the $\beta_k,\, k\in \N_0$, from our $\psi$ function. From this we get
 $$
 \sum_{k=0}^\infty|\beta_k| \stackrel{\text{Lemma}}= \frac{1}{\tilde{\phi}(0)} = \frac{1}{\phi(1)}   = \frac{1}{1-K}<\infty.
 $$
 We conclude that \eqref{formal} is a meaningful family of equations. In other words, $(X_n -\mu_X)$ can be represented as a moving-average process with respect to the white-noise sequence $(u_n)$.
\end{proof}
\noindent From the explicit construction in the proof of Theorem~\ref{existence}, we find the following branching representation of the INAR($\infty$) process. The branching formulation will be useful for the derivation of the moment-generating function. It furthermore summarizes the most elegant and, at the same time, efficient way for simulating from the INAR($\infty$) model:
\begin{proposition}\label{branching}
Let $\(X_n\)$ be an INAR($\infty$) sequence with respect to an immigration sequence $\(\varepsilon_i\)$ and reproduction coefficients $\alpha_k\geq 0,\, k\in\N,$ so that $\sum_{k=1}^\infty\alpha_k< 1$; see Definition~\ref{INAR}. Then
\begin{equation}
X_n = \sum\limits_{i\in\Z} \sum\limits_{j=1}^{\varepsilon_i}F^{(i,j)}_{n-i},\quad n\in\Z, \label{branching_rep}
\end{equation}
where 
 $\(F^{(i,j)}_n\)$ are independent (over $i\in\Z$ and $j\in\N$) copies of a branching process  $\(F_n\)$ defined by 
\begin{align}\label{F_n}
 F_n := \sum\limits_{g=0}^\infty G_{n}^{(g)},\quad n\in\Z. 
 \end{align}
The generations $(G_n)$ in \eqref{F_n} are constructed recursively by
\begin{align}\label{recursion}
 G_{n}^{(0)} &:=1_{\{n = 0\}} \quad\text{  and  }\quad G_{n}^{(g)}:=\sum_{k=1}^n \alpha_k\circ G_{n-k}^{(g-1)}:= \sum_{k=1}^n \sum_{m=1}^{G_{n-k}^{(g-1)}} \xi^{(n,k,g)}_m,\quad n\in\Z,
\end{align}
with  $\xi^{(n,k,g)}_m\sim\text{Pois}(\alpha_k)$ independently over $m,n,k,g$. Furthermore, we have the following distributional equality for the generic family-process $(F_n)$:
\begin{equation}
 \(F_n\)_{n\in\Z}\; \eqd \;\(1_{\{n = 0\}} + \sum\limits_{i=1}^n\sum\limits_{j=1}^{G_i^{(1)}}F^{(i,j)}_{n-i}\)_{n\in\Z}\label{recursivity}.
 \end{equation}
\end{proposition}
\begin{proof}
Equation \eqref{branching_rep} together with \eqref{F_n} and \eqref{recursion} is exactly the construction of a solution to the defining difference-equations \eqref{DE} in the proof of Theorem~\ref{existence}; see \eqref{F_n1} and \eqref{G_n_2}. To establish \eqref{recursivity}, consider the process on the right-hand side:
$$
\(\tilde{F}_n\)_{n\in\Z} \;:= \; \(1_{\{n = 0\}} + \sum\limits_{i=1}^n\sum\limits_{j=1}^{G_i^{(1)}}F^{(i,j)}_{n-i}\)_{n\in\Z}$$
We show that the process $\(\tilde{F}_n\)$ is constructed by the same (stochastic) recursion as $(F_n)$; see \eqref{F_n} and \eqref{recursion}. Then the equality in distribution follows. 
For $n\in\Z$, we define
\begin{align}
\tilde{G}^{(0)}_n&:= 1_{\{n = 0\}}\quad \text{   and   }\quad \tilde{G}^{(g)}_n := \sum\limits_{i=1}^n\sum\limits_{j=1}^{{G}_i^{(1)}}{G}^{(g-1,i,j)}_{n-i},\quad g\in\N, 
\end{align}
where $\(G^{(g,i,j)}_n\),\, g\in\N_0,$ are the generation processes that constitute the family processes $\(F^{(i,j)}_n\),\, i\in\Z,\, j\in\N,$ in \eqref{recursivity}. In particular, $\(G^{(g,i,j)}_n\)$ are independent copies of the generations $\(G^{(g)}_n\)$ defined in \eqref{recursion}.
Then, by construction, $\tilde{F}_n = \sum_{g\geq 0} \tilde{G}^{(g)}_n,\, n\in\Z$. This establishes a representation for $\(\tilde{F}_n\)$ of the same form as \eqref{F_n} for $\(F_n\)$. Next, we show that the summands $\(\tilde{G}^{(g)}_n\)$ follow the same recursion \eqref{recursion} as the original generations $\({G}^{(g)}_n\)$: for $g=0$, we have that $\tilde{G}^{(0)}_{n} = 1_{\{n = 0\}},\, n\in\Z$. So the starting value of the recursion for $\(\tilde{G}^{(g)}_n\)$ is the same as the starting value of the recursion \eqref{recursion} for $\({G}^{(g)}_n\)$. For $g=1$, recursion \eqref{recursion}  is also analogue:
\begin{align*}
\tilde{G}^{(1)}_n& = \sum\limits_{i=1}^n\sum\limits_{j=1}^{G_i^{(1)}}G^{(0,i,j)}_{n-i}= \sum\limits_{i=1}^n\sum\limits_{j=1}^{G_i^{(1)}}1_{\{n-i = 0\}}= G_n^{(1)}= \sum\limits_{k=1}^n\alpha_k\circ 1_{\{n-k=0\}}
=\sum\limits_{k=1}^n\alpha_k\circ \tilde{G}^{(0)}_{n-k}.
\end{align*}
And, for any $g\geq2$, we find that
$$
\tilde{G}^{(g)}_n = \sum\limits_{i=1}^n\sum\limits_{j=1}^{G_i^{(1)}} G^{(g-1,i,j)}_{n-i}
=\sum\limits_{i=1}^n\sum\limits_{j=1}^{G_i^{(1)}}\sum\limits_{k=1}^{n-i}\alpha_k\circ G^{(g-2,i,j)}_{n-i-k}
=\sum\limits_{k=1}^n\sum\limits_{i=1}^{n-k}\sum\limits_{j=1}^{G_i^{(1)}}\alpha_k\circ G^{(g-2,i,j)}_{n-i-k},
$$
where in the third equality we use that $G^{(g-2,i,j)}_{n-i-k} = 0,\, g\geq 2,\, n-i-k \leq 0$.
At this point we avoid the explicit representation of the counting sequences. We just remind ourselves that all thinnings involved are independent and establish
$$
\tilde{G}^{(g)}_n = \sum\limits_{k=1}^n\alpha_k\circ\sum\limits_{i=1}^{n-k}\sum\limits_{j=1}^{G_i^{(1)}} G^{(g-2,i,j)}_{n-i-k} =  \sum\limits_{k=1}^n\alpha_k\circ\tilde{G}^{(g-1)}_{n-k},\quad n\in\Z, \, g\geq 2.
$$
In other words, the processes $\(G_n\)$ and $\(\tilde{G}_n\)$ and, consequently, the processes $\(F_n\)$ and $\(\tilde{F}_n\)$ are constructed by the same stochastic recursion. We conclude that $\big(F_n\big)\eqd\(\tilde{F}_n\)$. This establishes \eqref{recursivity}.
\end{proof}
\subsection{Moment structure}
\noindent From the representation of the INAR($\infty$) sequence as a superposition of shifted i.i.d.\ family processes in Proposition~\ref{branching} above, we derive equations for the joint moment-generating function of the model. First, we fix some notation:
\begin{definition}\label{mgf}
For any sequence $(t_n)_{n\in\N_0}\subset \R$, let  $\mathrm{supp}\big((t_n)\big):=\min\{n\in\N_0:\, t_k=0, k>n\}$ be the \emph{support of the sequence $(t_n)$}. We use the convention that $\min\emptyset := \infty$. Furthermore, for $A\subset\R$, let
$c_{00}(A):= \big\{(t_n)_{n\in\N_0}\subset A: \, \mathrm{supp}\big((t_n)\big)< \infty \big\}$, the space of sequences in $A$ with a finite number of nonzero values. For any time series $(Y_n)_{n\in\N_0}$, we define the \emph{joint moment-generating function} 
\begin{align}
M_{(Y_n)} \big((t_n)_{n\in\N_0}\big) := \E \exp\left\{\sum\limits_{n=0}^\infty t_n Y_n \right\},\quad (t_n)\in c_{00}\( \R\).
\end{align}
\end{definition}
The somewhat unusual definitions above have been chosen for most direct comparison between the INAR($\infty$) joint moment-generating function and the Laplace functional of a Hawkes process; see Proposition~\ref{hawkes_laplace} and Proposition~\ref{approximating_laplace} below.

\begin{theorem} \label{INAR_mg}
Let  $(X_n)$ be an INAR($\infty$) sequence with respect to immigration parameter $\alpha_0\geq 0$ and reproduction coefficients $\alpha_k\geq 0,\, k\in\N$, such that $K=\sum_{k=1}^\infty\alpha_k<1$; see Definition~\ref{INAR}. Then there exists a constant $\delta >0$ such that
\begin{align}
M_{(X_n)} \big((t_n) \big)\leq\exp\bigg\{d\alpha_0  \frac{1+K}{2K} \bigg\}<\infty, \quad \(t_n\) \in c_{00} \big((-\infty, \delta] \big) \label{M_X_bound},
\end{align}
where $d:=\mathrm{supp}\big((t_n)\big) + 1$ is the maximal number of nonzero values of $(t_n)$.
Furthermore, we have that
\begin{align}
M_{(X_n)} \big((t_n) \big)& = \exp\left\{\alpha_0\sum\limits_{i\in\Z} \bigg(M_{(F_n)}\Big((t_{i+n})_{n\in\N_0}\Big)-1\bigg)\right\} \label{M_X},
\end{align}
where we set $t_m:=0$ for $m\leq 0$. 
Here, $\(F_n\)$ denotes the generic family-process from Proposition~\ref{branching}.
Its joint moment-generating function $M_{(F_n)}$ is the unique solution to 
\begin{align}
&M_{(F_n)}\big((s_n)\big)= \e^{s_0}\exp\left\{\sum\limits_{k=1}^\infty \alpha_k\Big(M_{(F_n)}\big((s_{n+k})_{n\in\N}\big)-1\Big)\right\} \label{M_F},\quad (s_n)\in c_{00}((-\infty,\delta])
.\end{align}
\end{theorem}

\begin{proof} For~\eqref{M_X}, we first apply representation~\eqref{branching_rep}:
$$
M_{(X_n)}\((t_n)\)=
\E\exp\left\{\sum\limits_{n=0}^\infty t_n X_n\right\}=
\E\exp\left\{\sum\limits_{n=0}^d t_n X_n\right\} = 
\E\exp\left\{\sum\limits_{n=0}^d t_n \sum\limits_{i\in\Z} \sum\limits_{j=1}^{\varepsilon_i}F^{(i,j)}_{n-i} \right\} 
{=} \E\prod\limits_{i\in\Z}\prod\limits_{j=1}^{\varepsilon_i}\exp\left\{\sum\limits_{n=0}^d t_n  F^{(i,j)}_{n-i}\right\} .$$
In the following, we set $t_m:= 0$ whenever $m<0$. Conditioning on the immigration sequence $\(\varepsilon_i\)$ and exploiting its independence from the family processes $\(F_n^{(i,j)}\)$
yields 
\begin{align}
M_{(X_n)}(t_1,\dots,t_d)&=  \E\prod\limits_{i\in\Z}\prod\limits_{j=1}^{\varepsilon_i} \E\[\exp\left\{\sum\limits_{n=0}^d t_n  F^{(i,j)}_n\right\}\]\nonumber\\
&=\prod\limits_{i\in\Z} \E\[ \E\[\exp\left\{\sum\limits_{n\in\Z} t_n  F_{n-i}\right\}\]^{\varepsilon_i}\]\nonumber\\
&= \prod\limits_{i\in\Z} \E\Big[ M_{(F_n)}\big((t_{i+n})_{n\in\N_0}\big)^{\varepsilon_i}\Big]\nonumber\\
&= \prod\limits_{i\in\Z} \exp\left\{\alpha_0 \Big(M_{(F_n)}\big((t_{i+n})_{n\in\N_0}\big) -1\Big)\right\}\nonumber\\
&= \exp\left\{\sum\limits_{i\in\Z}\alpha_0 \Big(M_{(F_n)}\big((t_{i+n})_{n\in\N_0}\big) -1\Big)\right\}\label{M_X_proof},
\end{align}
where in the last but one step we applied the formula for the probability-generating function of a Poisson random variable. Up to now, \eqref{M_X_proof} is only a formal representation of $M_{(X_n)}$ in terms of $M_{(F_n)}$. It is not clear yet, whether and when
$M_{(X_n)}$ is finite. 
For~\eqref{M_F}, we apply the equation~\eqref{recursion} of $\(F_n\)$ from Proposition~\ref{branching}:
$$
M_{(F_n)}\big((s_n)\big) = \E\exp\left\{\sum\limits_{n=0}^{\infty}s_{n} F_{n}\right\}
\stackrel{\eqref{recursivity}}{=}\E\exp\left\{\sum\limits_{n=0}^{{\infty}}s_{n}\( 1_{\{n = 0\}}+ \sum\limits_{k=1}^n\sum\limits_{j=1}^{G_k^{(1)}}F^{(k,j)}_{n-k}\)\right\}
$$
We note that in the last term, the index $k$ may run to $\infty$ instead of $n$, because
$F^{(k,j)}_{n-k}  = 0,\, j\in\N$, whenever $k>n$. After straightforward calculations, we obtain
\begin{align}
M_{(F_n)}\big((s_n)_{n\in\N_0}\big) =e^{s_0}\E\exp\left\{\sum\limits_{n=0}^{{\infty}}\sum\limits_{k=1}^\infty\sum\limits_{j=1}^{G_k^{(1)}}s_{n}F^{(k,j)}_{n-k}\right\}
=e^{s_0}\exp\left\{\sum\limits_{k=1}^\infty \alpha_k \Big(M_{(F_n)}\big((s_{n+k})_{n\in\N_0}\big)-1\Big)\right\}\label{M_F_proof}.
\end{align}
Next we derive finiteness of $M_{(F_n)}$. Let $(s_n)$ be a sequence with finite support and $s:=\max\{s_n\}$ so that
$\sum_{n=0}^{\infty}s_nF_n$
is bounded from above by $sS$, where $S:=\sum_{n=0}^\infty F_n$ denotes the total number of individuals in the generic family $\(F_n\)$. We remind ourselves of the defining equation \eqref{F_n} for the family process $(F_n)$ and find that
\begin{align}
S = \sum_{n=0}^\infty F_n =  \sum_{n=0}^\infty \sum _{g = 0}^\infty G^{(g)}_n= \sum _{g = 0}^\infty \sum_{n=0}^\infty G^{(g)}_n.\label{S}
\end{align}
 We denote the total number of individuals in the $g$-th generation by $Y_g :=  \sum_{n=0}^\infty G^{(g)}_n,\, g\in\N_0$. The sequence $\(Y_g\)_{g\in\N_0}$ is the embedded generation process. Applying \eqref{recursion}, we find that $Y_0 = 1$ and, for $g\geq 2$,
\begin{align*}
Y_g   =  \sum_{n=0}^\infty G^{(g)}_n 
=  \sum_{n=0}^\infty \sum\limits_{k=1}^\infty \alpha_k\circ G^{(g-1)}_{n-k}
=\sum\limits_{k=1}^\infty  \sum_{n=0}^\infty \alpha_k\circ G^{(g-1)}_{n-k}
\eqd  \sum\limits_{k=1}^\infty \alpha_k\circ \sum_{n=0}^\infty  G^{(g-1)}_{n-k}
=\sum\limits_{k=1}^\infty  \alpha_k\circ Y_{g-1}
= \sum\limits_{k=1}^{Y_{g-1}}\xi^{(g)}_k,
\end{align*}
where $\xi^{(g)}_k\iidsim \text{Pois}(K),\, k,g\in\N$. In other words, the embedded generation process $\(Y_g\)$ is a standard Galton--Watson branching process. From \eqref{S}, we see that $S\eqd \sum_{g= 0}^\infty Y_g $. In other words, $S$ is distributed like the cumulative limit of a standard Galton--Watson process. The moment-generating functions of such limits have been considered in the literature: as $K<1$, by Theorem 2.1 in \citet{nakayama04}, there exists a $\delta > 0$ such that $\E\exp\{\delta S\}<\infty$ if and only if there exists a $\tilde{\delta} > 0$ such that $\E\exp \big\{\tilde{\delta} \xi_1^{(1)}\big\}<\infty$. The latter is indeed the case because the moment-generating function of a Poisson variable is finite on $\R$. Furthermore, by Jensen's inequality, we see that $1\leq \lim_{n\to\infty} \E\exp\{\delta S/n\} \leq \lim_{n\to\infty} \(\E\exp\{\delta S\}\)^{1/n} =1$. So, for any given $\epsilon>0$, we can assume the existence of a $\delta>0$ such that $\E\exp\(\delta S\)<1+\epsilon$. In particular, 
\begin{align}
\exists\, \delta>0\text{  such that  }M_{(F_n)}\big((s_n)\big)\leq  \E\exp\(\delta S\)   < 1 + (1-K)/(2K)\quad \text{  for  }(s_n)_{n\in\N_0}\in c_{00}\((-\infty,\delta]\),\label{M_F_bound}
\end{align}where, as before, $K=\sum_{k=1}^\infty \alpha_k< 1$. We now have established finiteness of $M_{(F_n)}$ in a neighborhood of zero. It remains to establish finiteness of $M_{(X_n)}$. Our goal is to bound the series representation \eqref{M_X_proof} of $M_{(X_n)}$. To that aim, we need to refine the bound  \eqref{M_F_bound} for $M_{(F_n)}\big((s_n)\big)$.
To that aim, with the constant $\delta>0$ from \eqref{M_F_bound}, for $i\in\Z$, we introduce the sequences $\(\delta_n^{(i)}\)_{n\in\Z}$ defined by 
\begin{equation}
\delta_n^{(i)} :=
\begin{cases}
 \delta, &n = i,\, i-1,\, \dots ,\,i-d,\\
  0, &\text{else},
\end{cases}\label{delta_seq}
\end{equation}
 where $d = \mathrm{supp}((t_n)) +1$, as before, with $(t_n)$ the considered argument sequence.
Note that, for $i<0$, we have that
$\delta_n^{(i)} = 0,\,n\in\N_0$. 
Consequently, by definition of $M_{(F_n)}$, for $i<0$, we have that $M_{(F_n)}\(\(\delta_n^{(i)}\)_{n\in\N_0}\) = 1$. Furthermore, observe that $\(\delta_{n+k}^{(i)}\)_n = \(\delta_{n}^{(i-k)}\)_n,\, i,k\in\Z.$
For $(t_n)\in c_{00} \big((-\infty,\delta]\big)$, we have by component-wise monotonicity of $M_{(X_n)}$ that
\begin{align}
M_{(X_n)} \big((t_n) \big)\,\leq\, M_{(X_n)}\((\delta^{(d)} _{n+i})_{n\in\N_0}\)
\stackrel{\eqref{M_X_proof}}{=}  \exp\left\{\sum\limits_{i \in\Z}\alpha_0 \Big(M_{(F_n)}\big((\delta^{(d)} _{n+i})_{n\in\N_0}\big) -1\Big)\right\}
&{=}  \exp\left\{\sum\limits_{i=-\infty}^{d} \alpha_0 \bigg(M_{(F_n)}\Big((\delta^{({d-i})} _{n})_{n\in\N_0}\Big) -1\bigg)\right\}\nonumber\\
& = \exp\left\{ \alpha_0\sum\limits_{i=0}^{\infty} m_i \right\},
\label{M_X_bound}
\end{align}
where we set $m_i:= M_{(F_n)}\big((\delta^{({i})}_n)_{n\in\N_0}\big) -1,\,i\in\Z$. Note that, by \eqref{M_F_bound}, we have that 
\begin{align}
 0\leq  m_i \leq  (1-K)/(2K),\quad i\in\Z.\label{m_i_bounds}
 \end{align}
 and, in particular, $m_i = 0$ for $i<0.$
 In the following, we only consider $m_i$ with $i> d$. In this case, we get from \eqref{delta_seq} that $e^{\delta^{({i})} _{0}} =e^0=1$ and we obtain the recursion \begin{align}
m_i =  M_{(F_n)}\big((\delta^{({i})} _{n})_{n\in\N_0}\big) -1\stackrel{\eqref{M_F_proof}}{=} e^{\delta^{({i})} _{0}}\exp\left\{\sum\limits_{k=1}^\infty \alpha_k \Big(M_{(F_n)}\big((\delta^{({i})} _{n+k})_{n\in\N_0}\big)-1\Big)\right\} -1\stackrel{\eqref{delta_seq}}{=} \exp\left\{\sum\limits_{k=1}^{i} \alpha_k m_{i-k}\right\} -1,\quad i> d.\label{recursive_ub}
\end{align}
For the sum in the exponential, we find that 
$$
\sum_{k=1}^i \alpha_k m_{i-k}  \stackrel{\eqref{m_i_bounds}}{\leq} \sum_{k=1}^\infty \alpha_k (1-K)/(2K) = K(1-K)/(2K)<1,\quad i\in\Z.
$$ 
Therefore, we may apply the exponential inequality $\exp(x) \leq (1-x)^{-1},\, x<1,$ in \eqref{recursive_ub}:
\begin{align}
m_i \leq \frac{ 1}{1-\sum_{k=1}^i \alpha_k m_{i-k} } -1  = \frac{\sum_{k=1}^i \alpha_k m_{i-k} }{1-\sum_{k=1}^i \alpha_k m_{i-k} }\stackrel{\eqref{m_i_bounds}}{\leq} \frac{\sum_{k=1}^i \alpha_k m_{i-k} }{1-\sum_{k=1}^i \alpha_k (1-K)/(2K)}\leq \frac{2 \sum_{k=1}^\infty \alpha_k m_{i-k} }{1+K },\quad i> d. \label{M_i_ineq}
\end{align}
Summing both sides of \eqref{M_i_ineq} over $i> d$, we obtain
\begin{align}
\sum\limits_{i = d+1}^{\infty} m_i
\leq \frac{2}{1+K} \sum\limits_{i = d+1}^{\infty} \sum\limits_{k=1}^\infty \alpha_k m_{i-k}
=  \frac{2}{1+K} \sum\limits_{k=1}^\infty  \alpha_k  \sum\limits_{i = d+1}^{\infty}m_{i-k} 
\stackrel{\eqref{m_i_bounds}}{\leq}  \frac{2}{1+K}  \sum\limits_{k=1}^\infty  \alpha_k  \sum\limits_{l = 0}^{\infty}m_{l} 
= \frac{2K}{1+K}  \( \sum\limits_{l = 0}^{d}m_{l}  +  \sum\limits_{l = d+1}^{\infty}m_{l} \).\label{sum_ineq}
\end{align}
Keeping in mind that $1- {2K}/({1+K})> 0$, we solve \eqref{sum_ineq} for $ \sum_{i = d+1}^{\infty}m_{i}$ and find that
\begin{align}
\sum\limits_{i = d +1}^{\infty} m_i  \leq  \(1- \frac{2K}{1+K}\)^{-1}  \frac{2K}{1+K}   \sum\limits_{l = 0}^{d}m_{l} \stackrel{\eqref{m_i_bounds}}{\leq} \frac{2K}{1-K} d\frac{1-K}{2K}= (1 +d).\label{m_i_sum_bound}
\end{align}
For the summation of $(m_i)_{i\in\N_0}$ over $i\in\N_0$, we finally obtain
\begin{align}
\sum\limits_{i = 0}^{\infty} m_i  \stackrel{\eqref{m_i_sum_bound}}{\leq}  \sum\limits_{i = 0}^{d} m_i+d +1 \stackrel{\eqref{m_i_bounds}}{\leq} (d+1)\(\frac{1-K}{2K}+1\)= (d+1) \frac{1+K}{2K}. \label{m_i_sum}
\end{align}
 We conclude that, for all $(t_n)\in c_{00}( (-\infty,\delta])$ with $\mathrm{supp}((t_n)) = d$, we have that
 \begin{align}
M_{(X_n)}(t_1,\dots,t_d) \stackrel{\eqref{M_X_bound}}{\leq} \exp\bigg\{\alpha_0\sum\limits_{i = 0}^{\infty} m_i\bigg\}\stackrel{\eqref{m_i_sum}}{\leq}
 \exp\bigg\{\alpha_0d \frac{1+K}{2K} \bigg\} <\infty.
\end{align}
\par Uniqueness of $M_{(F_n)}$ follows by induction over the (finite) support of the argument sequence. In that sense, the implicit equation \eqref{M_F_proof} specifies $M_{(F_n)}$ recursively.
\end{proof}
\noindent As a matter of fact, every moment-generating function that is finite in a neighborhood of zero has a Taylor series about zero. The coefficients of this series are the joint moments. Consequently, from Theorem~\ref{INAR_mg}, we obtain
\begin{corollary}
Let $\(X_n\)$ be an INAR($\infty$) sequence as in Theorem~\ref{existence}. For $m\in\N$ and $k_1,\dots,k_m \in \N_0,$ we have that
$$
\E\bigg[X_1^{k_1}X_2^{k_2}\cdot\cdot\cdot X_m^{k_m}\bigg] < \infty.
$$
\end{corollary}
\noindent The second moments of the INAR($\infty$) sequence are particularly tractable:
\begin{proposition}\label{autocovariance}
Let $(X_n)$ be an  INAR($\infty$) process with reproduction coefficients $\alpha_k\geq 0,\, k\in\N$, such that $K:=\sum_{k=1}^\infty\alpha_k < 1$, and immigration parameter $\alpha_0\geq 0$. Furthermore, let $R(j):= \Cov\(X_n, X_{n+j}\),\,j\in\Z,$ be the autocovariance function of the (stationary) sequence.
Then
\begin{align}
R(j)&= \frac{\alpha_0}{1-K}\sum_{k=0}^\infty\beta_{k}\beta_{k+|j|} \geq 0,\quad j\in \Z,\label{R(j)}\end{align}
where
$
\beta_0:= 1$ and $\beta_k:= \sum_{i=1}^k \alpha_i \beta_{k-i}$. In addition, we have that
\begin{align}
\sum_{j=0 }^\infty R(j) \leq \frac{\alpha_0}{\(1-K\)^3}<\infty. \label{covsum_bound}
\end{align}
\begin{proof}
With the notation from the moving-average representation of the INAR($\infty$) sequence in Proposition~\ref{MA}, we find that, for $j,n\in\Z$, 
$$
R(j) =  \Cov\(X_n-\mu_X, X_{n+j}-\mu_X\) =  \Cov\( \sum\limits_{k=0}^\infty \beta_k u_{n-k},  \sum\limits_{k=0}^\infty \beta_k u_{n+j-k}\).
$$
From Proposition~\ref{AR}, we know that $\Cov\(u_n, u_{n+j}\) = 1_{\{j = 0\}}\alpha_0/(1-K),\, j\in\Z,$.  Furthermore, from Proposition~\ref{MA}, we have that the coefficients $\beta_k$ are absolutely summable. So~\eqref{R(j)} follows from Proposition 3.1.2.\ in \citet{brockwell91}. 
For the sum of the autocovariance sequence, we observe
\begin{align*}
\sum\limits_{k=0}^{\infty} R(k) =  \frac{\alpha_0}{1-K}\sum\limits_{k=0}^\infty\sum\limits_{i = 0}^\infty\beta_i\beta_{i+k}
= \frac{\alpha_0}{1-K}\sum\limits_{i=0}^\infty \beta_i\sum\limits_{k = 0}^\infty\beta_{i+k}
\leq  \frac{\alpha_0}{1-K}\sum\limits_{i=0}^\infty \beta_i\sum\limits_{k = -i}^\infty\beta_{i+k}
=\frac{\alpha_0}{\(1-K\)^3}.
\end{align*}
The last equality re-uses the result $\sum_{i=0}^\infty \beta_i=1/(1-K)$ from Proposition~\ref{MA}.
\end{proof}
\end{proposition}
\subsection{INAR($p$) embedding}
\noindent As one would expect, the INAR($\infty$) model can be interpreted as the limit of the usual INAR($p$), $p<\infty,$ model: if we truncate the excitement sequence of an INAR($\infty$) time series after a large lag $p$, the resulting process is an approximation of the original time series:
\begin{proposition}
Let  $(X_n)$ be an INAR($\infty$) sequence with respect to immigration parameter $\alpha_0\geq 0$ and reproduction coefficients $\alpha_k\geq 0,\, k\in\N$, $K=\sum_{k=1}^\infty\alpha_k<1$; see Definition~\ref{INAR}. Furthermore, for $p\in\N$, let  $\(X^{(p)}_n\)$ be another INAR($\infty$) sequence with respect to the same immigration parameter $\alpha_0$ as above and reproduction coefficients $\(\alpha^{(p)}_k\)$, specified by
$$
\alpha^{(p)}_k :=
\begin{cases}
\alpha_k,&1\leq k\leq p\\
0,& k>p.
\end{cases}
$$
Then the finite-dimensional distributions of $\(X^{(p)}_n\)$ converge to the finite-dimensional distributions of $\(X_n\)$ as $p$ goes to infinity.
\end{proposition}
\begin{proof}
It suffices to show that the corresponding moment-generating functions converge. We apply the notation from Theorem~\ref{INAR_mg}, respectively, from Definition~\ref{mgf}. Let $(t_n)_{n\in\N_0}$ be any sequence in $c_{00}((-\infty,\delta])$, where $\delta>0$ is the constant from~\eqref{M_X_bound}. We will establish that for any $\lim_{p\to\infty} M_{\(X^{(p)}_n\)} (t_1,\dots,t_d) = M_{(X_n)} (t_1,\dots,t_d)$. As a first step, we show that the moment-generating functions of the corresponding family-processes converge, i.e., that $\lim_{p\to\infty} M_{\(F^{(p)}_n\)} \big((s_n)_{n\in\N_0}\big) = M_{(F_n)} \big((s_n)_{n\in\N_0}\big)$ for all sequences $(s_n)_{n\in\N_0}\in c_{00}\big( (-\infty,\delta]\big)$. We argue with induction over the size of the support of the sequence: for a sequence
 $(s_n^{(0)})_{n\in\N_0}\subset (-\infty,\delta)$ with $\mathrm{supp}(s_n^{(0)}) =0$, we find that
 $$
\lim\limits_{p\to\infty}M_{(F_n^{(p)})} \((s^{(0)}_n)_{n\in\N_0}\) \stackrel{\eqref{M_F}}{=} \e^{s^{(0)}_0} = M_{(F_n)} \((s^{(0)}_n)_{n\in\N_0}\). 
$$
Now, for  $d\in\N$, let $\(s_n^{(d)}\)_{n\in\N_0}\subset (-\infty,\delta]$ be a sequence such that $\mathrm{supp}\(\(s_n^{(d)}\)\)\leq d$. We obtain
\begin{align*}
\lim\limits_{p\to\infty}M_{(F_n^{(p)})} \((s^{(d)}_n)_{n\in\N_0}\) 
&\stackrel{\eqref{M_F}}{=}    \e^{s^{(d)}_0}\exp\left\{\lim\limits_{p\to \infty} \sum\limits_{k=1}^\infty \alpha^{(p)}_k\Big(M_{(F_n^{(p)})}\big((s^{(d)}_{n+k})_{n\in\N_0}\big)-1\Big)\right\}.\\
 \end{align*}
Clearly, $\lim_{p\to \infty} \alpha^{(p)}_k = \alpha_k$. Furthermore, we see that, for $k\in\N$, $s^{(d)}_{n+k} = 0$ whenever $n>d-k$. This means, the support of the sequences $(s^{(d)}_{n+k})_n$ is strictly smaller than $d$. Consequently, by induction hypothesis, we have that, for $k\in\N$, $\lim_{p\to\infty}M_{(F_n^{(p)})} \((s^{(d)}_{n+k})_{n\in\N_0}\)=M_{(F_n)} \((s^{(d)}_{n+k})_{n\in\N_0}\)$. We may conclude: 
 \begin{align*}
 \lim\limits_{p\to\infty}M_{(F_n^{(p)})} \((s^{(d)}_n)_{n\in\N_0}\) 
    =  e^{s^{(d)}_0}\exp\left\{\sum\limits_{k=1}^\infty \alpha_k\Big(M_{(F_n)}\big((s^{(d)}_{n+k})_{n\in\N_0}\big)-1\Big)\right\}
   \stackrel{\eqref{M_F}}{=}   M_{(F_n)}\((s^{(d)}_n)_{n\in\N}\),\quad d\in\N.
 \end{align*}
 This establishes weak convergence of generic family processes. 
In a second step, we apply formula \eqref{M_X} for $M_{X^{{(p)}}}$:
\begin{equation}
 M_{(X_n^{(p)})}\((t_n)\) = \exp\left\{\sum\limits_{i\in\Z} \alpha_0\Big(M_{(F_n^{(p)})}\((t_{i +n})_{n\in\N_0}\)-1\Big)\right\}\label{M_X_p},\quad (t_n)\in c_{00}((-\infty,\delta]).
\end{equation}
We show that we may apply the dominated convergence theorem on the right-hand side of \eqref{M_X_p}: let $\(\delta^{(i)}_n\)_{n\in\Z},\, i\in\Z,$ be sequences defined by $\delta^{(i)}_n:= \delta,\, n\in\{i-d ,\dots,i-1,i\},$ and $\delta^{(i)}_n:= 0,\, \text{else}$; see \eqref{delta_seq}. We find, for $i\in\Z$, that
 $$
M_{(F_n^{(p)})}\Big(\big(t_{n+i}\big)_{n\in\N_0}\Big)-1\;\leq\; M_{(F_n^{(p)})}\Big(\big( \delta^{(d-i)}_n\big)_{n\in\N_0} \Big)-1 \;\leq\;
 M_{(F_n)}\Big( \big(\delta^{(d-i)}_n\big)_{n\in\N_0}  \Big)-1,\quad p\in\N.
 $$
 The second inequality follows by similar induction over the support size as before. 
We already know from~\eqref{m_i_sum} that
$
 \Big(M_{(F_n)}\big( (\delta^{(d-i)}_n) \big)-1\Big)_{i\in\Z}
$
is summable over $i$. In view of this dominating summable sequence, we take the limit on both sides of ~\eqref{M_X_p} and conclude that, for $ (t_n)\in c_{00}\big((-\infty,\delta]\big)$,
\begin{eqnarray*}
\lim\limits_{p\to \infty} M_{(X_n^{(p)})}\((t_n)\) &
= &\exp\left\{\lim\limits_{p\to \infty}  \sum\limits_{k\in\Z} \alpha_0\big(M_{(F_n^{(p)})}\((t_{n+k})_{n\in\N_0}\big)-1\)\right\}\\
&\stackrel{\text{DC}}{=}&\exp\left\{ \sum\limits_{k\in\Z} \alpha_0\big(M_{(F_n)}\((t_{n+k})_{n\in\N_0}\big)-1\)\right\} \\
&=&  M_{X}\((t_n)\).
\end{eqnarray*}
\end{proof}

\section{The Hawkes process}
\noindent After the long first section on the new INAR($\infty$) model, the following shorter section formally presents the well-known Hawkes point process. We treat point processes as random counting-measures and only consider point processes  on $\R$. First, we fix some general notation and terminology. Then we recall the definition, the existence theorem and selected properties of the (univariate) Hawkes process. For the general theory, we mainly follow \citet{resnick87}, Chapter 3. For the Hawkes part, our main references are the seminal papers \citet{hawkes71a,hawkes71b} and \citet{hawkes74}.
\subsection{Preliminaries}\label{Preliminaries}
\noindent Let $\B:=\B(\R)$ be the Borel-sets in $\R$ and $
\B_b:=\{B\in\B(\R):\,B\ \mathrm{ bounded}\}.
$
A measure $m$ on $\R$ is a \emph{point measure} if $m(B)\in\N_0,\ B\in\B_b.$   We denote the space of point measures on $\R$ by $M_p:=M_p(\R)$. Let $C^+_K:=C^+_K(\R)$  be the space of nonnegative continuous functions on $\R$ with compact support.
Point measures $ (m _n)$ \emph{converge vaguely} to a point measure $m$ if $\lim_{n\to\infty}  \int f(t) m_n(\mathrm{d}t)\to \int f(t) m(\mathrm{d}t)$, $f\in C^+_K(\R)$; we write $m_n\vto m$. Vague convergence yields the \emph{vague topology} on $M_p$. 
The Borel $\sigma$-algebra generated by this topology, $\mathcal{M}_p:=\mathcal{B}(M_p)$, coincides with the $\sigma$-algebra generated by the sets $\{ m\in\mathcal{M}_p:\  m(A)=k\},\, A\in\mathcal{B}_b, k\in\N_0$; see Lemma 1.4.\ in \citet{kallenberg83}. 
Any measurable mapping $\Phi: (M_p,\mathcal{M}_p) \to (\R, \mathcal{B})$ such that $\lim_{n\to\infty}\Phi(m _n) = \Phi( m )$ whenever $ m _n \vto m $, is \emph{continuous with respect to the vague topology}. Our basic underlying probability space is $(\Omega, \mathcal{F},\P)$. A measurable mapping $
 N: (\Omega, \mathcal{F})\rightarrow \(M_p,\mathcal{M}_p\),\, \omega\mapsto N(\omega)$ is called \emph{point process}. The \emph{history} of a point process $N$ is the filtration $\(\mathcal{H}^{N}_t\)$, where, for $t\in\R$,
\begin{align}
\mathcal{H}^{N}_t := 
\sigma\bigg(\Big \{
\omega\in\Omega:\,
 N(\omega)\big(
(a,b]
\big)
 \in B \Big\}
:\,-\infty<a < b \leq t, 
B\subset \N_0\bigg).
\label{history}
\end{align}
We assume that $\mathcal{H}^{N}_t \subset \mathcal{F},\, t\in\R$. Note that our definition of a point process allows multiple points, i.e., we may have that ``$\P\[N(\{t\}) > 1|N(\{t\})>0 \] >0$''. Also note that, for $t\in\R$, the $\sigma$-algebra $\mathcal{H}^{N}_t $ includes all sets of the form
\begin{align}
\Big\{\omega:\, N(\omega)\big(\{t_n\}\big)=k_n,\, N(\omega)\big((t_{n-1}, t_{n})\big) = 0 ,\, n= 0, -1,-2,\dots\Big\} \bigg\} \quad \text{with}\quad  t=: t_0 \geq t_{-1} \geq \dots \quad \text{and} \quad k_0, k_{-1}, \dots \strut\in\N.
\end{align}
\subsection{Definition and existence}
\begin{definition}\label{Hawkes intensity}
For any point measure $m \in M_p$, define the \emph{Hawkes intensity}
$$
\lambda(t| m ):=\eta + \int_{\R}h(t-s) m (\mathrm{d}s),\quad t\in\R,
$$
where $\eta>0$ is a constant and 
$h: \R\rightarrow\R_0^+$ is a nonnegative measurable function with $h(t)=0,\ t\leq 0$. We refer to $\eta$ as \emph{immigration intensity} and to $h$ as \emph{reproduction intensity}.
\end{definition}
The immigration intensity is often called \emph{baseline intensity} and the reproduction intensity is often called \emph{excitement function}. However, our objective is to highlight the similarity between Hawkes and INAR processes. Consequently, we make use of a joint branching-process terminology; see Definition~\ref{INAR}.

\begin{definition}\label{hawkes_definition}
Let $\lambda$ be a Hawkes intensity as in Definition~\ref{Hawkes intensity}.
A \emph{Hawkes process} is a point process $N$ that is a solution to the family of equations
\begin{equation}
\E\bigg[1_A N\Big(\big(a,b\big]\Big)\bigg] = \E\[1_A\int\limits_{a}^b\lambda(s|N)\mathrm{d}s\], \quad a < b , \,A\in\mathcal{H}^N_a. \label{HE1}
\end{equation}
\end{definition}
A priori, it is not clear whether this family of equations has a solution, whether any possible solution would be unique (in a distributional sense) and whether the distribution of a solution would be stationary. These questions are answered by the following proposition; see \citet{hawkes74}:
\begin{proposition} \label{uniqueness}
Let $\lambda$ be a Hawkes intensity with immigration intensity $\eta>0$ and reproduction intensity $h$ such that $\int_0^\infty h(t) \mathrm{d}t < 1$. Then there is precisely one stationary process that satisfies \eqref{HE1}.
\end{proposition}
\noindent The existence and uniqueness result above is established by the observation that the solution to \eqref{HE1} must be a cluster process or---more specifically---a branching process with immigration: the points are interpreted as individuals that are either immigrants or offspring. The immigrants (or cluster centers) stem from a homogeneous Poisson process with intensity $\eta$. These immigrants form generation zero of the following branching procedure: an immigrant at time $s\in\R$ triggers an inhomogeneous Poisson process with intensity $h\(\cdot - s\)$ where $h$ is the reproduction intensity of the process as in Definition~\ref{Hawkes intensity}. These offspring individuals form generation one. Each of these first-generation individuals again triggers an inhomogeneous Poisson process in a similar way, etc., so that the families (or clusters) are generated by cascades of inhomogeneous Poisson processes.  

\subsection{The Laplace functional}
\noindent The cluster and branching process point of view is also fertile beyond the results of Proposition~\ref{uniqueness}. For example, it leads to equations for the \emph{Laplace functional} of a Hawkes process. The Laplace functional $\Psi_N$ of a point process $N$ is a functional defined on the space of nonnegative measurable functions with compact support by
$$
\Psi_N[f] := \E\exp\left\{-\int_{-\infty}^\infty f(t)N\(\mathrm{d}t\)\right\}.
$$
The next proposition is Theorem 2 in \citet{hawkes74}---with a slight modification as the original statement refers to the probability generating functional whereas we prefer the nowadays more common Laplace functional notion:
\begin{proposition} \label{hawkes_laplace}
Let  $N$ be a Hawkes process with immigration intensity $\eta> 0$ and reproduction intensity $h$ as in Proposition~\ref{uniqueness}. Then the Laplace functional $\Psi_N$ of $N$ is
\begin{align*}
\Psi_N[f] &= \exp\left\{\eta \int_{-\infty}^\infty\Big(\Psi_F\big[\,f(t + \cdot)\big] - 1\Big)\,\mathrm{d}t\right\},
\end{align*}
where $\Psi_F$ is a functional that is the unique solution to
\begin{align*}
\Psi_F[f] &=  e^{-f(0)}\exp\left\{ \int_{0}^\infty \Big(\Psi_F\big[\,f(t + \cdot)\big] - 1\Big)h(t)\,\mathrm{d}t\right\}.
\end{align*}
In both equalities, $f$ denotes an arbitrary measurable, nonnegative function with compact support.
\end{proposition}
\section{Links between INAR($\infty$) and Hawkes processes}\label{Links}
\noindent In the following section, we first explain how discrete-time INAR($\infty$) processes can approximate a continuous-time Hawkes process. After the convergence theorem, we establish a number of properties of the approximating sequence. Finally, we collect some structural analogies of the two models.
\subsection{Preliminaries}
\noindent Let $Y_n,\,n\in\N,$ and $Y$ be random variables with values in some topological space. The sequence $(Y_n)$ \emph{converges weakly} to $Y$ if $\lim_{n\to\infty}\E \varphi(Y_n)=\E\varphi(Y)$ for all nonnegative continuous bounded functions $\varphi$. We define weak convergence of point processes with respect to the vague topology $\mathcal{M}_p$ on $M_p$; see Section~\ref{Preliminaries}. In this case, weak convergence of point processes is equivalent to convergence of their finite-dimensional distributions; see \citet{daley03}, Theorem 11.1.VII. 
General weak convergence theory, as developed in the monograph \citet{billingsley68}, considers sequences in metric spaces.
Therefore it is important to note that \emph{the vague topology is metrizable}; see \citet{resnick87}, Proposition 3.17. In other words, we may treat $\(M_p,\mathcal{M}_p\)$ as a metric space where necessary. A most helpful theorem in the weak-convergence context is the \emph{continuous mapping theorem}; see Theorem 5.1 in \citet{billingsley68}. 
We apply it in the following form:
\begin{proposition} \label{discontinuity}
 Let $\(N_n\)$ and $N$ be point processes such that $N_n\wto N,\, n\to\infty.$
Furthermore, let $f: \R\to\R_0^+$ be a bounded, measurable function with compact support and with a set of discontinuities $D_f\in\mathcal{B}$ such that $\P\[N\(D_f\) >0 \]=0$. Then, $\int f(t)N_n(\mathrm{d}t) \wto \int f(t)N(\mathrm{d}t) $ for $n\to\infty.$
\end{proposition}

\subsection{The convergence theorem}
\noindent Next to the conditions on the reproduction intensity $h$ from Definition~\ref{hawkes_definition} and Proposition~\ref{uniqueness}, we introduce an additional assumption: we want $h$ piecewise continuous. We say a function $f:\R\to\R$ is \emph{piecewise continuous} if its set of discontinuities $D_f\subset\R$ is finite and for all $t_0\in D_f$ the limits $\lim_{t\to t_0^-}f(t)$ and $\lim_{t\to t_0^+}f(t)$ exist and are finite. Combining all assumptions on $h$ yields the following important technical
 \begin{lemma} \label{K_tilde}
Let $h:\R\to\R_{\geq 0}$ be piecewise continuous function with $h(t) = 0,\, t\leq 0$, and $\int h\,\mathrm{d}t < 1$. Then there exist constants $\delta>0$ and $\tilde{K}<1$ such that, for any $\Delta\in (0, \delta)$,
\begin{align}
K^{(\Delta)}:= \Delta \sum_{k=1}^{\infty}h(k\Delta)\leq \tilde{K}<1 \label{K_unif_bound}.
\end{align}
\end{lemma}
\noindent In the sequel of the section, let $\delta>0$ and $K^{(\Delta)}\leq \tilde{K}<1,\, \Delta\in(0,\delta),$ be as in the lemma above. We state the main mathematical result of our paper:
\begin{theorem}\label{weak_convergence}
Let $N$ be a Hawkes process with immigration intensity $\eta$ and reproduction intensity $h$ as in Lemma~\ref{K_tilde}. For $\Delta \in(0,\delta)$, let $\(X^{(\Delta)}_n\)$ be an INAR($\infty$) sequence with immigration parameter $\Delta\eta$ and reproduction coefficients $\Delta h(k\Delta),\, k\in\N$. From these sequences, we define a family of point processes by 
\begin{equation}
N^{(\Delta)}(A) := \sum\limits_{k: k\Delta\in{A}}X^{(\Delta)}_k,\quad A\in\B,\, \Delta \in(0,\delta)\label{derived_pp}.
\end{equation}
Then, we have that
$$
N^{(\Delta)}\stackrel{\mathrm{w}}{\longrightarrow} N\quad\text{for}\quad \Delta\rightarrow 0.
$$
\end{theorem}
\noindent Our proof uses the standard weak-convergence approach---as followed in the Hawkes context, e.g.,\ by \citet{bremaud96}. First, tightness of the approximating family is established. By Prohorov's theorem, tightness yields weak subsequential limits for all subsequences. Then we show that all those potential weak subsequential limits have the same distribution as the Hawkes process. This will establish the result. An alternative approach would be convergence of Laplace functionals that are given by  
\begin{proposition}\label{approximating_laplace}
For some $\Delta\in(0,\delta)$, let $N^{(\Delta)}$ be as in Theorem~\ref{weak_convergence}.
Let $f$ be a nonnegative measurable function with compact support. Then the Laplace functional of $N^{(\Delta)}$ evaluated at $f$ is
\begin{align*}
&\Psi_{N^{(\Delta)}}[f]= \exp\left\{\Delta\eta\sum\limits_{i\in\Z} \bigg(\Psi^{(\Delta)}_{F^{(\Delta)}}\Big(\big(f((i+n)\Delta)\big)_{n\in\N_0}\Big)-1\bigg)\right\}.
\end{align*}
Here, the function $\Psi^{(\Delta)}_{F^{(\Delta)}}$ operating on sequences $(s_n)_{n\in\N_0}\in c_{00}\([0,\infty)\)$ is a solution to
\begin{align*}
\Psi^{(\Delta)}_{F^{(\Delta)}}\big((s_n)_{n\in\N_0}\big) 
&= e^{-s_0}\exp\left\{\sum\limits_{k=1}^\infty \Delta h(k\Delta)\Big(\Psi^{(\Delta)}_{F^{(\Delta)}}\big((s_{k+n})_{n\in\N_0}\big)-1\Big)\right\}.
\end{align*}
\end{proposition}
\begin{proof}
Plugging in definitions, we obtain
\begin{align*}
\Psi_{N^{(\Delta)}}[f] = \E\exp\left\{-\int_\mathbb{R} f(t)N^{(\Delta)}\(\mathrm{d}t\)\right\}
&= \E\exp\left\{-\sum_{n \in\Z} X^{(\Delta)}_nf(n\Delta)\right\}.
\end{align*}
By stationarity of $(X^{(\Delta)}_n)$, we may assume without loss of generality that $f(t) = 0,\, t< 0.$ We apply formulas~\eqref{M_X} and~\eqref{M_F} for the joint moment-generating function of the INAR($\infty$) process $\(X^{(\Delta)}_n\)$ and the corresponding generic family process $\(F^{(\Delta)}_n\)$:
  $$\Psi_{N^{(\Delta)}}[f]  = M_{\(X^{(\Delta)}_n\)}\Big(\big(-f(0), -f(\Delta), -f(2\Delta), \dots\big)\Big) \stackrel{\eqref{M_X}}{=} \exp\left\{\alpha_0\sum\limits_{i\in\Z} \Bigg(M_{(F^{(\Delta)}_n)}\bigg(\Big(-f\big((n+i)\Delta\big)\Big)_{n\in\N_0}\bigg)-1\Bigg)\right\}.
  $$
 We set 
  $
  \Psi_{F^{(\Delta)}}^{(\Delta)}\big((s_n)\big):=M_{\(F^{(\Delta)}\)}\big((-s_n)\big), \, \(s_n\)\in c_{00}\big( [0,\infty)\big)$.
  This establishes the lemma.
\end{proof}
\noindent The similarities between the formulas in Proposition~\ref{approximating_laplace} above and the corresponding equations for the Hawkes process in Proposition~\ref{hawkes_laplace} are striking. Still, rather than establishing the convergence result from Theorem~\ref{weak_convergence} via the Laplace functionals, we choose a more direct reasoning on the process level that contains useful information on the approximating point process family as a by-product. Some of the properties that are necessary for this convergence proof are collected in the following lemmas:
\begin{lemma}\label{expectation}
For any $\Delta\in(0,\delta)$, let $N^{(\Delta)}$ be a point process as in Theorem~\ref{weak_convergence}. Then, for $A\in\mathcal{B}$, we have that
\begin{align}
A\cap \{k\Delta:\, k\in\Z\}  = \emptyset 
\quad \Rightarrow \quad
N^{(\Delta)}\big(A \big)\stackrel{}{=} 0,\; \text{almost surely} .\label{nullsets}
\end{align}
For the expectation, we find that
\begin{align}
\E N^{(\Delta)} \big(\{k\Delta\}\big) = \Delta \frac{\eta}{1- K^{(\Delta)}} <  \delta \frac{\eta}{1- \tilde{K}},\quad k\in\Z,\ \label{EN(kDelta)}
\end{align}
and, for $a<b$,
\begin{align}
\E N^{(\Delta)}\big([a,b]\big) <(b-a+2\delta)\frac{\eta}{1-\tilde{K}} <\infty. \label{EN[a,b]}
\end{align}
\end{lemma}

\begin{proof}
Claims \eqref{nullsets} and \eqref{EN(kDelta)} directly follow from the definition of $N^{(\Delta)}$ in \eqref{derived_pp} together with Proposition~\ref{existence} and Lemma~\ref{K_tilde}. For \eqref{EN[a,b]}, we find that the number of grid points in the interval $[a,b]$ is less or equal $\lceil(b-a)/\Delta\rceil + 1$. To get rid of the ceiling function, we observe that
$\lceil(b-a)/\Delta\rceil + 1 < (b-a)/\Delta + 2$. 
 Combining this with the facts that $K^{(\Delta)}\leq\tilde{K}$ and $\Delta <\delta$ together with \eqref{nullsets} and \eqref{EN(kDelta)} yields inequality \eqref{EN[a,b]}.
\end{proof}
\noindent For the second moments, we find 
\begin{lemma} \label{unifvar}
For $\(N^{(\Delta)}\)_{\Delta\in(0,\delta)}$, the approximating family of point processes from Theorem~\ref{weak_convergence}, we have that
$$
\sup\limits_{\Delta\in(0,\delta)} \Var\(N^{(\Delta)}(A)\) <\infty, \quad A\in\mathcal{B}_b.
$$
\end{lemma}
\begin{proof}
Let $A\in\mathcal{B}_b$ be a bounded Borel set.
Note that
\begin{align*}
\Var \(N^{(\Delta)}(A)\)=\Var\(\,\sum\limits_{n:\,n\Delta\in A}X^{(\Delta)}_n\)
=\sum\limits_{m:\,m\Delta\in A}\sum\limits_{n:\,n\Delta\in A}\Cov\(X^{(\Delta)}_n,X^{(\Delta)}_m\)
=\sum\limits_{m:\,m\Delta\in A}\sum\limits_{n:\,n\Delta\in A}R^{(\Delta)}(n-m),\end{align*}
where $R^{(\Delta)}$ is the autocovariance function of the INAR($\infty$) sequence from Proposition~\ref{autocovariance}. From this proposition, we know that $R^{(\Delta)}(k)\geq 0,\,k\in\Z$ and $\sum_{k=0}^\infty R^{(\Delta)}(k) \leq \eta\Delta(1-K^{(\Delta)})^{-3}$. Applying these results yields
\begin{align*}
&\Var \( N^{(\Delta)}(A)\)\\
&\leq\sum\limits_{m:\,m\Delta\in A}\sum\limits_{n\in\Z} R^{(\Delta)}(n-m)
\leq\sum\limits_{m:\,m\Delta\in A}\frac{2\eta\Delta}{(1-K^{(\Delta)})^3}
\leq \(2+\frac{\sup A - \inf A}{\Delta}\)\frac{2\eta\Delta}{(1-\tilde{K})^3}
\leq\(2\delta +\sup A - \inf A\) \frac{2\eta}{(1-\tilde{K})^3},
\end{align*}
where $\tilde{K}<1$ does not depend on the choice of $\Delta\in(0,\delta)$; see Lemma~\ref{K_tilde}.
\end{proof}
\noindent A family of random variables $\(Y_i\)_{i\in I}$ is \emph{uniformly integrable} if
$
\lim_{M\to \infty} \sup_{i\in I} \E\[1_{|Y_i|>M}|Y_i|\] = 0.
$
We obtain uniform integrability of the random variables in question as a corollary from Lemma~\ref{unifvar} above:
\begin{lemma}\label{unifint}
Let $\(N^{(\Delta)}\)_{\Delta\in(0,\delta)}$ be the approximating family of point processes from Theorem~\ref{weak_convergence} and $A\in\mathcal{B}_b$. Then we have that the family of random variables $\(N^{(\Delta)}(A)\)_{\Delta\in(0,\delta)}$ is uniformly integrable.
\end{lemma}
\noindent A family of probability measures $\(\P^{(i)}\)_{i\in I}$ on $\(M_p,\sigma\(\mathcal{M}_p\)\)$ is \emph{uniformly tight} if, for all $\varepsilon>0$, there exists a compact set $K\subset M_p$ such that
$\P^{(i)}[K^c]<\varepsilon,\, i\in I$.
\begin{lemma}\label{tightness}
The family of the probability measures $\(\P^{(\Delta)}\)_{0<\Delta< \delta}$ on $\(M_p, \sigma\(\mathcal{M}_p\)\)$ corresponding to the random point processes
$
\(N^{(\Delta)}\)_{0<\Delta< \delta}
$
is uniformly tight.
\end{lemma}
\begin{proof}
The claim follows with Proposition 11.1.VI.\ from \citet{daley03} if for all compact intervals $[a,b]\subset\R$ and for all $\epsilon>0$ there exists an $M<\infty$ such that
$$
 \sup\limits_{\Delta \in (0,\delta)}\P\bigg[N^{(\Delta)}\big([a,b]\big)>M\bigg]<\epsilon.
$$
 The uniform boundedness of these probabilities is a consequence of Lemma~\ref{expectation} and Markov inequality: for any $\epsilon > 0$ and $a < b$, let $M_{\epsilon}:=   (b-a+2\delta)\eta /(1-\tilde{K})$, where $\delta>0$ and $\tilde{K}<1$ as in Lemma~\ref{K_tilde}. Then we have that
$$
\P\bigg[N^{(\Delta)}\big([a,b]\big)>M_{\epsilon}\bigg]\leq \frac{\E N^{(\Delta)}\big([a,b]\big)}{M_{\epsilon} }<
 (b-a+2\delta)\frac{\eta}{M_{\epsilon} (1-\tilde{K})}= \epsilon,\quad \Delta\in(0,\delta).
$$
\end{proof}

\begin{proof}\emph{(Theorem~\ref{weak_convergence})}
As a consequence of Lemma~\ref{tightness}, the family of probability measures $\(\P^{(\Delta)}\)_{\Delta\in(0,\delta)}$ that corresponds to the family of point processes
$
\(N^{(\Delta)}\)_{\Delta\in(0,\delta)}$ is relatively compact for weak convergence by Prohorov's theorem; see \citet{daley03}, Theorem A.2.4.I. So every sequence in $\(\P^{\Delta}\)_{\Delta\in(0,\delta)}$, respectively, 
$\(N^{(\Delta)}\)_{\Delta\in(0,\delta)}$, contains a weakly convergent subsequence.
In particular, for any zero sequence in $(0,\delta)$, we can find a subsequence $( \Delta_n)$ such that $\(N^{\(\Delta_n\)}\)$ converges weakly to some point process $N^*$.
If the distribution of $N^*$ does not depend on the initial choice of the subsequence, it follows that the original sequence converges weakly to $N^*$; see Theorem 2.3.\ in \citet{billingsley68}.
Reconsider the implicit defining-equation \eqref{HE1} 
from Definition~\ref{hawkes_definition}. By Proposition~\ref{uniqueness}, we know that this equation determines the distribution of the solving process. So, for the proof of Theorem~\ref{weak_convergence}, it suffices to show that any subsequential limit candidate $N^*$ solves \eqref{HE1}. Furthermore, one can show that it suffices to prove \eqref{HE1} for $A^*\in\frak{B}_a^{N^*}$, where $\frak{B}_a^{N^*}$ is a semiring of sets that generates the $\sigma$-algebra $\mathcal{H}^{N^*}_a$; see~\eqref{history}. A semiring is a class of sets $\mathcal{A}$ such that for any pair $A,B \in \mathcal{A} $ one has (i) $A\cap B \in  \mathcal{A}$ and (ii) $(A\cup B) \setminus (A\cap B) = \cup_{i=1}^nA_i$ for some $n\in\N$, $\(A_i\)\subset\mathcal{A}$ and $A_i\cap A_j =\emptyset,\, i,j =1,\dots,n$. We consider, for any $a\in\R$ and any point process $N$,
\begin{align}
\frak{B}_a^{N}:= &\bigg\{
\big \{ \omega\in\Omega:\, N\((s_1,t_1]\)(\omega)\in D_1,\dots, N\((s_k,t_k]\)(\omega) \in D_k 
\big\}
: \;-\infty < s_i<t_i\leq a,\, D_i\subset \N_0,\, k\in\N\bigg\}\label{semiring}.
\end{align}
One can check that the set system $\frak{B}_a^{N}$ is indeed a semiring. 
Summarizing the above, for the proof of Theorem~\ref{weak_convergence}, it suffices to establish 
\begin{align}
\E\bigg[1_{A^*} N^*\big((a,b]\big)\bigg] = \E\[1_{A^*}\int\limits_{a}^b\lambda(s|N^*)\,\mathrm{d}s\], \quad a < b , \,A^*\in\mathcal{H}^{N^*}_a. \label{HE}
\end{align}
First, let us establish a discrete version of~\eqref{HE} for the approximating sequence: set $N_n:= N^{(\Delta_n)}$ for all $\Delta_n$ in the chosen subsequence. For $a <b$ and $A_n  \in \frak{B}_a^{N_n}$, we find that
\begin{align*}
\E\bigg[1_{A_n}N_n\big((a,b]\big)\bigg] &= \E\[1_{A_n} \sum\limits_{k:\, k\Delta_n\in(a,b]}X_{k}^{( \Delta_n)}\] \\
&=  \E\[1_{A_n} \sum\limits_{k:\, k \Delta_n\in(a,b]} \(\varepsilon^{( \Delta_n)}_k+ \sum\limits_{l=1}^\infty \big( \Delta_n h(l \Delta_n)\big)\circ X_{k-l}^{( \Delta_n)}\)\] \\
&=  \E\[1_{A_n} \sum\limits_{k:\, k \Delta_n\in(a,b]} \( \Delta_n\eta+ \sum\limits_{l=1}^\infty \Delta _nh(l \Delta_n)X_{k-l}^{( \Delta_n)}\)\],\quad n\in\N.
\end{align*}
The last step follows by the observation that the immigrations $\varepsilon^{(\Delta_n)}_k$ as well as the thinnings that contribute to  $X^{(\Delta_n)}_k$, $k\Delta_n > a$, are independent of   $X^{(\Delta_n)}_{k-1},X^{(\Delta_n)}_{k-2},\dots $\; Rewriting the inner sum of the last term as an integral with respect to the random measure $N_n$, we obtain, for $a < b$,
\begin{align}
\E\bigg[1_{A_n} N_n\big((a,b]\big)\bigg]
= \E\[1_{A_n} \sum\limits_{k:\, k\Delta_n\in(a,b]}\Delta_n \(\eta+\int_{-\infty}^{k\Delta_n} h(k\Delta_n - s) N_n(\mathrm{d}s)\)\]\label{discrete_eq},\quad  A_n\in \frak{B}_a^{N_n},\,n\in\N.
\end{align}
Note that here and throughout the proof the upper integration bounds in the Hawkes intensities do not require special attention due to the assumption $h(0) = 0$ for reproduction intensities $h$ in Definition~\ref{Hawkes intensity}, respectively, Lemma~\ref{K_tilde}.
Now we show that \eqref{discrete_eq} converges to \eqref{HE} corresponding to the Hawkes process. For both sides of equation \eqref{discrete_eq}, this is achieved in three steps:
\begin{itemize}
\item First, we establish that the random variable in the expectation can be written as $\Phi(N_n)$, where
$\Phi: (M_p,\mathcal{M}_p) \to(\R,\B) $
denotes some measurable mapping with set of discontinuities $D_{\Phi}\subset M_p$.
\item Next, we show that $\P\[ N^*\in D_{\Phi}\] = 0$.
\item Finally, we prove that the random variables in question are uniformly integrable.
\end{itemize}
By Proposition~\ref{discontinuity}, the first two points together imply that $\Phi(N_n)\wto\Phi(N^*)$. The additional uniform-integrability property yields that the corresponding expectations also converge; see Theorem 5.4 in \citet{billingsley68}.\\

\noindent\emph{Left-hand side of~\eqref{discrete_eq}:}\\
Consider the map
\begin{align}
\Phi: \(M_p,\mathcal{M}_p\) \to \(\R,\mathcal{B}\),\  m  \mapsto1_{\{ m ((s_1,t_1]) \in D_1, \dots, m ((s_k,t_k]) \in D_k \}}  m \big((a,b]\big);\label{continuity}
\end{align}
see the definition of $\frak{B}_a^N$ in \eqref{semiring} for the notation.
We claim that $\Phi$ is vaguely continuous on $M_p \setminus \{m:\, m\(D_{\Phi}\)>0 \}$, where $D_{\Phi}:=\(\{a,b\}\cup\bigcup_{i=1}^k\{s_i, t_i\}\)$. Indeed: the map $m \mapsto m((a,b])$ is vaguely continuous on $M_p\setminus\{ m :\,  m (\{a,b\})>0\}$ and, for $ i=1,\dots,k,$ the maps
$m \mapsto m ((s_i, t_i]\})$ are vaguely continuous on $M_p\setminus\{ m :\,  m\( \{s_i, t_i\}\)>0\}$.
 The map $\N_0^k\ni\(l_1,\dots,l_k\)\mapsto 1_{\{l_1 \in D_1,\, \dots ,\,l_k \in D_k \}}$ is trivially continuous, so that
$ m  \mapsto 1_{\{ m ((s_1,t_1]) \in D_1, \dots, m ((s_k,t_k]) \in D_k \}} $ 
 is continuous on $M_p\setminus\bigcup_{i=1}^k\{s_i,t_i\}$.
From Proposition~\ref{discontinuity}, we have that $\Phi(N_n)\wto\Phi(N^*)$ if 
$\P\[N^*\(D_{\Phi}\)>0\] = 0$. Because $D_{\Phi}$ is finite, it suffices to show that $\P\[N^*\big(\{t\})>0\] = 0 \text{ for any } t\in \R.$ 
\begin{align*}
\P\big[N^*(\{t\})>0\big] = \E 1_{N^*\(\{t\}\)>0} \leq  \E N^*(\{t\})\stackrel{\text{Lemma}~\ref{unifint}}{=}\lim \limits_{n\to\infty}\E N_n(\{t\})\stackrel{\eqref{EN(kDelta)}}{<} \Delta\frac{\eta}{(1-\tilde{K})},\quad t\in\R, \Delta\in(0,\delta).
\end{align*}
So $\P\[N^*(\{t\})>0\] = 0 $ and, consequently, $\P\[N^*(D_{\Phi})>0\]$ must also be zero.
This establishes $\Phi(N_n)\wto\Phi(N^*)$, respectively, $1_{A_n} N_n(a,b) \wto 1_{A^*} N^{*}(a,b)$.
From Lemma~\ref{unifint}, we know that $\(N_n(a,b)\)$ is uniformly integrable, so $\(1_{A_n} N_n(a,b)\)$ is also uniformly integrable. Combining weak convergence and uniform integrability yields convergence of expectations
$$
\lim\limits_{n\to\infty}\E\bigg[1_{A_n} N_n\big((a,b]\big)\bigg] = \E\bigg[1_{A^*} N^{*}\big((a,b]\big)\bigg] .
$$
We have established the convergence of the left-hand side of~\eqref{discrete_eq} to the left-hand side of~\eqref{HE}.\\

\noindent\emph{Right-hand side of~\eqref{discrete_eq}:}\\
Note that the right-hand side of~\eqref{discrete_eq} converges to the right-hand side of \eqref{HE} if
\begin{align}
\sum\limits_{k\in(a,b]} \Delta_n\E&\[1_{A_n} \int\limits_{-\infty}^{k\Delta_n} h(k\Delta_n - s) N_n(\mathrm{d}s)\]\ \stackrel{n\to\infty}{\longrightarrow}\ \int\limits_a^b\E\[1_{A^*} \int\limits_{-\infty}^{t} h(t - s) N^{*}(\mathrm{d}s)\]\mathrm{d}t.\label{rhs}
\end{align}
As a first step for establishing \eqref{rhs}, note that, for all choices of M with $-M<a$, and, for $t\in[a,b]$,
\begin{align}
1_{A_n} \int_{-M}^{t} h(t - s) N_n(\mathrm{d}s)\ \wto\
1_{A^*} \int_{-M}^{t} h(t - s) N^*(\mathrm{d}s),\quad n\to \infty. \label{truncated}
\end{align}
This is due to a continuous-mapping argument similar to the one we have used for the left-hand side of~\eqref{discrete_eq}. We establish that the variances of the random variables $\int_{-M}^{t} h(t - s) N_n(\mathrm{d}s),\, n\in\N_0,$ are uniformly bounded:
\begin{align*}
\Var \int_{-M}^{t} h(t - s) N_n(\mathrm{d}s)
\leq \Var \sum\limits_{l=1}^{\lceil M/\Delta_n\rceil} h(l\Delta_n) X^{(\Delta_n)}_{-l}
= \sum_{l=1}^{\lceil M/\Delta_n\rceil} \sum_{m=1}^{\lceil M/\Delta_n\rceil} h(l\Delta_n) h(m\Delta_n)  \Cov \(X^{(\Delta_n)}_{-l}, X^{(\Delta_n)}_{-m}\).
\end{align*}
At this point, we write $R^{(\Delta_n)}$ for the autocovariance function of the INAR($\infty$) process $\(X^{(\Delta_n)}_l\)$. Applying Proposition~\ref{autocovariance} yields
\begin{align}
&\Var \int_{-M}^{t} h(t - s) N_n(\mathrm{d}s)\nonumber\\
&\leq \sum_{l=1}^{\lceil M/\Delta_n\rceil} h(l\Delta_n)  \sup h\sum_{m=1}^{\lceil M/\Delta_n\rceil}R^{(\Delta_n)}(|l-m|)\nonumber\\
& \leq  \sum_{l=1}^{\lceil M/\Delta_n\rceil} h(l\Delta_n)  \sup h\sum_{i=-\infty}^\infty R^{(\Delta_n)}(i)\nonumber\\
& \stackrel{\eqref{covsum_bound}}{\leq}  \sum_{l=1}^{\lceil M/\Delta_n\rceil} h(l\Delta_n) \sup h \frac{2\eta \Delta_n}{\big(1-K^{(\Delta_n)}\big)^3}\nonumber\\
& \leq\( \sup h\)^2\frac{M+1}{\Delta_n} \frac{2 \Delta_n\eta}{\big(1-K^{(\Delta_n)}\big)^3}\nonumber\\
&\stackrel{\eqref{K_unif_bound}}{\leq}\( \sup h\)^2 \frac{2(M+1)\eta}{\big(1-\tilde{K}\big)^3}\leq c_M<\infty,\label{variance_bound}
\end{align}
where $c_M$ is a constant independent of $n$, respectively, $\Delta_n$. We may conclude that the random variables $1_{A_n} \int_{-M}^{t} h(t - s) N_n(\mathrm{d}s),\, n\in\N,$ are uniformly integrable. Weak convergence together with uniform integrability yields convergence of expectations. We have established that, for $M$ with $-M<a$,
\begin{align}
\lim\limits_{n\to\infty}\E\[1_{A_n} \int_{-M}^{t} h(t - s) N_n(\mathrm{d}s)\] =
\E\[1_{A^*} \int_{-M}^{t} h(t - s) N^*(\mathrm{d}s)\] .\label{exp_truncated}
\end{align}
For the proof of~\eqref{rhs}, we consider a truncated part of $h$ and its remaining tail separately. For the truncated part, we use \eqref{exp_truncated}; for the tail part, we use the integrability condition $\int h\,\mathrm{d}t <1$: for any  $M >-a$, we have
\begin{align}
&\left|\sum\limits_{k\in(a,b]} \Delta_n\E\[1_{A_n} \int\limits_{-\infty}^{k\Delta_n} h(k\Delta_n - s) N_n(\mathrm{d}s)\]\nonumber
- \int\limits_a^b\E\[1_{A^*} \int\limits_{-\infty}^{t} h(t - s) N^{*}(\mathrm{d}s)\]\mathrm{d}t\right| \\
&\leq\left|\sum\limits_{k\in(a,b]} \Delta_n\E\[1_{A_n} \int\limits_{-M}^{k\Delta_n} h(k\Delta_n - s) N_n(\mathrm{d}s)\]
- \int\limits_a^b\E\[1_{A^*} \int\limits_{-M}^{t} h(t - s) N^{*}(\mathrm{d}s)\]\mathrm{d}t\right|\nonumber \\
&\hspace{0.5cm}+ \sum\limits_{k\in(a,b]} \Delta_n \E\[1_{A_n} \int\limits_{-\infty}^{-M}h(k\Delta_n - s) N_n(\mathrm{d}s)\]\nonumber\\
&\hspace{0.5cm}+\int\limits_a^b \E\[1_{A^*} \int\limits_{-\infty}^{-M} h(t - s) N^{*}(\mathrm{d}s)\]\mathrm{d}t.\label{bigsum}
\end{align}
Let $\varepsilon>0$.
We show that we can find $M_\varepsilon$ and $N_\varepsilon\in\N$ such that each of the three summands in \eqref{bigsum} is bounded by $\varepsilon/3$ for $n\geq N_\varepsilon$. First, consider the integrand of the last summand in \eqref{bigsum}. By arguing similarly to the m((a,b]) part in \eqref{continuity}, we find that $\E N^*(\mathrm{d}t) /\mathrm{d}t \leq \eta/(1-\tilde{K})$. So we can choose $M^{(1)}_{\varepsilon}>0$ so large that, for $t\in[a,b]$,
 \begin{align*}
  \E\[1_{A^*} \int_{-\infty}^{-M^{(1)}_{\varepsilon}} h(t - s) N^{*}(\mathrm{d}s)\]&\leq \E\int_{-\infty}^{-M^{(1)}_{\varepsilon}} h(t - s) N^*(\mathrm{d}s) \leq \frac{\eta}{1-\tilde{K}}\int_{M^{(1)}_{\varepsilon}+a}^{\infty} h(s)\mathrm{d}s< \frac{\varepsilon}{3(b-a)}.
 \end{align*}
 The summands of the second term in \eqref{bigsum} can be bounded by ${\varepsilon}/{(3\lceil b-a\rceil)}$ in a similar and even more direct way---uniformly over $n$ and possibly with respect to another $M^{(2)}_{\varepsilon}>0$. We set $M_\varepsilon := \max\big\{M^{(1)}_\varepsilon, M^{(2)}_\varepsilon\big\}$. So taking the integral over the interval $[a,b]$ of the last summand, respectively, the Riemann sum of the second summand in~\eqref{bigsum}, yields
\begin{align}
\int\limits_{a}^{b}  \E\[1_{A^*} \int_{-\infty}^{-M_{\varepsilon}} h_M(t - s) N^{*}(\mathrm{d}s)\] \mathrm{d}t <\frac{\varepsilon}{3} \label{epsilon_bound1},
\end{align}
respectively,
\begin{align}
\sum\limits_{k:\,k\Delta_n\in(a,b]} \Delta_n\E\[1_{A^*} \int_{-\infty}^{-M_{\varepsilon}} h_M(k\Delta_n - s) N^{*}(\mathrm{d}s)\] <\frac{\varepsilon}{3} \label{epsilon_bound2}.
\end{align}
For the first term in \eqref{bigsum}, denote 
$$
E_n(t):= \E\[1_{A_n} \int_{-M_{\varepsilon}}^{t} h(t - s) N_n(\mathrm{d}s)\]
\quad \text{
and
}
\quad
E^*(t):= \E\[1_{A^*} \int_{-M_{\varepsilon}}^{t} h(t - s) N^{*}(\mathrm{d}s)\].
$$
 From \eqref{truncated}, we already know that, for any choice of $M_{\varepsilon}$, $\lim_{n\to\infty}|E_n(t)-E^*(t)|=0,\, t\in(a,b]$. However, the convergence of the Riemann-like sums $\sum_{k:\,k\Delta_n \in(a,b]}E_n(k\Delta_n)\Delta_n$ to the integral $\int_a^b E^*(s)\mathrm{d}s$ is nontrivial as the functions $E_n$ are themselves part of the sequence. We write
\begin{align}
\left|\sum\limits_{k:\,k\Delta\in(a,b]}\Delta_n E_n(k\Delta_n) - \int_a^bE^*(t)\mathrm{d}t\right|\leq \left|\sum\limits_{k:\,k\Delta\in(a,b]}\Delta_n E_n(k\Delta_n) - \int_a^bE_n(t)\mathrm{d}t\right|+\left| \int_a^bE_n(t)\mathrm{d}t - \int_a^bE^*(t)\mathrm{d}t\right| \label{Riemann_sum}.
\end{align}
The second absolute difference in \eqref{Riemann_sum} converges to zero because of dominated convergence of $(E_n)$. Indeed, for $t\in[a,b]$,
 \begin{align}
 \E\[1_{A_n} \int_{-M_{\varepsilon}}^{t} h(t - s) N_n(\mathrm{d}s)\]\leq  \E \int_{-M_{\varepsilon}}^{b} h(t - s) N_n(\mathrm{d}s)
 \leq  \sup h \sum_{k:\,k\Delta_n\in (-M_{\varepsilon},b]} \E X^{(\Delta_n)}_k
  \leq \sup h\frac{\lceil M_{\varepsilon}+b\rceil\eta}{1-\tilde{K}}.\label{ub}
\end{align}
Note that $\sup h < \infty$ follows from the piecewise-continuity assumption. In view of the upper bound \eqref{ub}, we apply the dominated convergence theorem and choose $N_\varepsilon^{(1)}\in\N$ so large that, for the second absolute difference in~\eqref{Riemann_sum}, we have
\begin{align}
\left| \int_a^bE_n(t)\mathrm{d}t - \int_a^bE^*(t)\mathrm{d}t\right| <\frac{\varepsilon}{6},\quad n\geq N_\varepsilon^{(1)}. \label{epsilon_sechstel1}
\end{align}
For the first absolute difference in \eqref{Riemann_sum}, we assume that, without loss of generality, the piecewise continuous function $h$ is uniformly continuous on $(0,\infty)$. Otherwise, we note that any piecewise continuous function on $\R$ that is vanishing at infinity is uniformly continuous on each of its continuous pieces and do the following calculation once for every uniformly continuous piece of $h$. Uniform continuity gives us a constant $\delta_h>0$, so small, that, for any $t_0>0$, 
\begin{align}
 |t-t_0| <\delta_h\; \land \;t>0\quad\Rightarrow\quad
|h(t)-h(t_0)|<\frac{\varepsilon (1-\tilde{K})}{12\eta(b-a+\delta)(M_\varepsilon+b+\delta)}. \label{delta_h}
\end{align}
Now, choose $N_{\varepsilon}^{(2)}$ so large that 
\begin{align}
\Delta_n<\min\left\{\delta_h, \frac{\varepsilon(1-\tilde{K})}{12\eta (b-a+\delta)\sup h}\right\}\quad  \text{for  }\quad n\geq N_{\varepsilon}^{(2)}. \label{N_2}
\end{align} 
Here again, $\delta$ and $\tilde{K}$ are the constants from Lemma~\ref{K_tilde}.
Let $a\leq s<t\leq b$ with $t-s< \Delta_n (<\delta_h),$ then
\begin{eqnarray}
|E_n(t) - E_n(s)|&
=&\left|\sum\limits_{k\Delta_n\in (-M_\varepsilon,t)} \E\[1_{A_n} h(t - k\Delta_n) X_k^{(\Delta_n)}\] - \sum\limits_{k\Delta_n\in (-M_ \varepsilon,s]} \E\[1_{A_n} h(s - k\Delta_n) X_k^{(\Delta_n)}\] \right|\nonumber\\
&\leq& \(\sum\limits_{k\Delta_n\in (-M_\varepsilon,s]}| h(t - k\Delta_n)- h(s - k\Delta_n) |+\sum\limits_{k\Delta_n\in (s,t)} h(t - k\Delta_n) \)\E X_0^{(\Delta_n)}\nonumber\\
&\stackrel{\eqref{delta_h}}{\leq}& \(\frac{M_\varepsilon+s+\Delta_n}{\Delta_n}\frac{\varepsilon(1-\tilde{K})}{12\eta(b-a+\delta)(M_\varepsilon+b+\delta)}+\sup h \)\frac{\Delta_n\eta}{1-K^{(\Delta)}}\nonumber\\
&\leq& \frac{\varepsilon}{12(b-a+\delta)} +\Delta_n\frac{\eta\sup h}{1-\tilde{K}}\nonumber\\
&\stackrel{\eqref{N_2}}{\leq}& \frac{\varepsilon}{12(b-a+\delta)} + \frac{\varepsilon}{12(b-a+\delta)} \quad =\quad   \frac{\epsilon}{6(b-a+\delta)}. \label{bound2}
\end{eqnarray}
Summarizing the above calculation, we have established the existence of an $N_\varepsilon^{(2)}\in\N$ such that, for all $n\geq N_\varepsilon^{(2)}$, we have $|E_n(t) - E_n(s)|\leq  \varepsilon/(6(b-a+\delta))$ whenever $|t-s|<\Delta_n$ and $s,t \in[a,b]$.
For the first absolute difference in \eqref{Riemann_sum}, we therefore get
\begin{eqnarray}
\left|\sum\limits_{k:\,k\Delta_n\in(a,b]}\Delta_n E_n(k\Delta_n) - \int_a^bE_n(t)\mathrm{d}t\right|
&\leq &\sum\limits_{k:\,k\Delta_n\in(a,b]}\int_{k\Delta_n}^{(k+1)\Delta_n}\left| E_n(k\Delta_n) - E_n(t)\right|\mathrm{d}t\nonumber\\
&\leq& \sum\limits_{k:\,k\Delta_n\in(a,b]}\int_{k\Delta_n}^{(k+1)\Delta_n}\frac{\varepsilon}{6(b-a+\delta)}\mathrm{d}t\nonumber\\
&\stackrel{\eqref{bound2}}{\leq}&\frac{(b-a+\Delta_n)}{\Delta_n} \Delta_n \frac{\varepsilon}{6(b-a+\delta)}\nonumber\\
& \leq &\frac{\varepsilon}{6} \label{epsilon_sechstel2}, \quad n\geq N_\varepsilon^{(2)}.
\end{eqnarray}

Set $N_\varepsilon := \max\Big\{N_\varepsilon^{(1)},N_\varepsilon^{(2)}\Big\}$. Combining \eqref{epsilon_sechstel1} and \eqref{epsilon_sechstel2}, we get that
\begin{align}
\left|\sum\limits_{k:k\Delta\in(a,b]}\Delta_n E_n(t) - \int_a^bE^*(t)\mathrm{d}t\right|&<\frac{\varepsilon}{3},\quad n\geq N_\varepsilon.\label{epsilon_bound3}
\end{align}
Combining \eqref{epsilon_bound1}, \eqref{epsilon_bound2} and \eqref{epsilon_bound3}, shows that \eqref{bigsum} is smaller than the given $\varepsilon$, for $n\geq N_{\varepsilon}$ and $M:=M_{\varepsilon}$,\ i.e.,\begin{align*}
\lim\limits_{n\to\infty}\sum\limits_{k:k\Delta_n\in(a,b]}\Delta_n\E \[1_{A_n}\int_{-\infty}^{\Delta_n} h(k\Delta_n - s)N_n(\mathrm{d}s)\]=\int_a^b\E \[1_{A^*}\int_{-\infty}^{t} h(t - s) N^*(\mathrm{d}s)\]\mathrm{d}t.
\end{align*}
We have established that the right-hand side of \eqref{discrete_eq} also converges to the right-hand side of \eqref{HE}. With the result from Proposition~\ref{uniqueness} on the uniqueness property of $\eqref{HE}$, we find that every subsequential limit $N^*$ has the same distribution as the Hawkes process $N$. We may then conclude that, for $\Delta\to 0$, the approximating sequence of point processes $\(N^{\(\Delta\)}\)$ converges weakly to the Hawkes process $N$.
\end{proof}

\subsection{Structural analogies}
\noindent Besides the formal convergence result from Theorem~\ref{weak_convergence}, we point out a number of more general, partly obvious structural parallels between the Hawkes and the INAR($\infty$) model. The branching structure of the INAR($\infty$) model described after Theorem~\ref{existence} is the same as the branching structure of the Hawkes process described after Proposition~\ref{uniqueness}. This similar underlying structure yields analogous equations for the moment-generating function of the INAR($\infty$) and the Laplace functional of the Hawkes process; see Theorem~\ref{INAR_mg} and Proposition~\ref{hawkes_laplace}. Consequently, we can expect similar distributional properties. In the following statements, we compare the models more directly by considering a specific Hawkes process $N$ together with its approximating family of INAR($\infty$) sequences, $\(X^{(\Delta)}\),\,\Delta\in(0,\delta)$, obtained from Theorem~\ref{weak_convergence}. 
\begin{enumerate}
\item The defining equations \eqref{DE} and \eqref{HE} have similar interpretations: taking expectations conditional on $\mathcal{H}^{X^{(\Delta)}}_{n-1}:= \sigma\(X^{(\Delta)}_k:\, k\leq{n-1}\)$ on both sides of \eqref{DE} yields
$$
\frac{\E\[X_n^{(\Delta)}| \mathcal{H}^{X^{(\Delta)}}_{n-1} \]}{\Delta} = \eta +\sum\limits_{k=-\infty}^{n-1} h\big((n-k)\Delta \big)X_{k}^{(\Delta)},\quad n\in\Z,
$$
which is similar to the local version of \eqref{HE}
$$
\frac{\E\[ N(\mathrm{d}t)| \mathcal{H}^N_{t}\]}{\mathrm{d}t} = \eta + \int\limits_{-\infty}^{t}h(t-s)N(\mathrm{d}s),\quad t\in\R .
$$

\item The stability criteria from Theorems~\ref{existence} and~\ref{uniqueness} correspond: for the time series case, $\sum_{k=1}^\infty \Delta h(\Delta k)< 1$ is a sufficient existence condition which in the Hawkes case becomes $\int_0^\infty h(t)\mathrm{d}t < 1$. \citet{bremaud01} establishes the existence of a nontrivial Hawkes process, where the total weight of the reproduction intensity equals 1 and the immigration intensity is zero. The analogous statement for INAR($\infty$) sequences can be derived in a similar way.
\item From the correspondence of the generating functions, we know that the moments of the INAR($\infty$) and the Hawkes process must be similar. The first and second moments of both model classes can be calculated and compared explicitly. The first moments are
$$
\frac{\E X^{(\Delta)}_n}{\Delta} = \frac{\eta}{1-K^{(\Delta)}},
\quad \text{respectively ,} \quad \frac{\E N(\mathrm{d}t)}{\mathrm{d}t} = \frac{\eta}{1-K}.
$$
For the first equality see Theorem~\ref{existence}; for the second equality see~\citet{hawkes71a}.
\item
For the autocovariances of both models, i.e., for
$$
R^{(\Delta)}(n):= \frac{\E\[X^{(\Delta)}_0 X^{(\Delta)}_{n}\]}{\Delta^2} - \(\frac{\E X^{(\Delta)}_0}{\Delta}\)^2,\quad n\in\Z,
$$
respectively, for
$$
r(t) :=\frac{\E\[\mathrm{d}N_0 \mathrm{d}N_{t}\]}{(\mathrm{d}t )^2 } -\(\frac{\E N(\mathrm{d}t)}{\mathrm{d}t}\)^2, \quad t\in\R,
$$
we find at the origin
\begin{align}
\label{Var_INAR}
R^{(\Delta)}(0) = \frac{1}{\Delta} \frac{\eta}{1-K^{(\Delta)}} + \sum\limits_{k=1}^\infty\Delta h(k\Delta)R(k),
\end{align}
respectively,
\begin{align}
r(0) =  \frac{1}{\mathrm{d}t} \frac{\eta}{ 1-K} + \int\limits_{0^+}^{\infty}h(s)r(s)\mathrm{d}s.\label{Var_Hawkes}
\end{align}
Implicit equations of Yule--Walker type are valid in both cases:
\begin{align}
R^{(\Delta)}(n) = \sum\limits_{k=1}^{\infty}  h(\Delta k) R^{(\Delta)}(|n| - k)\Delta\label{INAR_YW},\quad n\neq 0,
\end{align}
respectively,
\begin{align}
r(t) =\int\limits_{0}^{\infty}h(s)r(|t|-s)\mathrm{d}s, \quad t\neq 0 .\label{Hawkes_YW}
\end{align}
Equations~\eqref{Var_INAR} and~\eqref{INAR_YW} are standard facts given the representation of the INAR($\infty$) sequence as a standard AR($\infty$) process in Proposition~\ref{AR}; \eqref{Var_Hawkes} and \eqref{Hawkes_YW} are derived  in \citet{hawkes71b}. 
\end{enumerate}
\hspace{0.5cm}
\subsection{The choice of the counting sequence distribution}\label{The Choice of The Counting Sequence Distribution}
\noindent As a last remark, we again refer to the choice of the counting sequence distribution in Definition~\ref{thinning_operator}. An obvious alternative to the Poisson distribution would have been the Bernoulli distribution; see the discussion after Definition~\ref{INAR}. With the Bernoulli choice, each individual would have not more than one offspring at each future point in time instead of potentially unboundedly many. We want to indicate that in the limit (in the sense of Theorem~\ref{weak_convergence} where all reproduction coefficients go to zero) this option would yield the same result: for $\Delta\in(0,1)$ and $(\alpha_k)\subset [0,1]$ such that $\sum\alpha_k<1$, let
$
\xi^{(\Delta)}_k \iidsim \text{Pois}(\Delta \alpha_k),\, k\in\N$, and $\tilde{\xi}^{(\Delta)}_k \iidsim \text{Bernoulli}(\Delta\alpha_k),\, k\in\N.
$
Then one can easily show that
$$
\lim\limits_{\Delta \to 0^+}\frac{ \P\[\xi^{(\Delta)}_k = 0\]}{\P\[\tilde{\xi}^{(\Delta)}_k = 0\]} =\lim\limits_{\Delta \to 0^+} \frac{ \P\[\xi^{(\Delta)}_k = 1\]}{\P\[\tilde{\xi}^{(\Delta)}_k = 1\]} = 1, \quad k\in\N.
$$
So when $\Delta$ is very small, the offspring distribution candidates, Poisson and Bernoulli, become very similar. Roughly speaking, the limiting procedure in Theorem~\ref{weak_convergence} is nothing else than a (complicated) superposition of limits of the form $\sum_{k=1}^\infty \xi^{(\Delta)}_k$. For these kinds of sums, we have that
$$
\lim\limits_{\Delta \to 0^+} \frac{\P\[\sum\limits_{k=1}^\infty \xi^{(\Delta)}_k= n\]}{\P\[\sum\limits_{k=1}^\infty \tilde{\xi}^{(\Delta)}_k= n\]} = 1,\quad n\in\N_0.
$$
If $\alpha_k>0$ infinitely often, then---by the Poisson limit theorem---the two considered probabilities are even equal for all $\Delta\in[0,1]$.
 In view of the above, one can expect that Theorem~\ref{weak_convergence} with the Bernoulli distribution as a starting point would yield the same limit as the Poisson distribution, namely the Hawkes process.
\section{Conclusion}\label{Discussion}
\noindent The mathematical formulation of the correspondence between INAR and Hawkes processes in Theorem~\ref{weak_convergence} has the following heuristic interpretation relevant for practical applications: let $\(N^{(\Delta)}\)$ be the approximating family of INAR($\infty$)-based point processes for a Hawkes process $N$ as in Theorem~\ref{weak_convergence}. Then for $\Delta>0$ (small), the finite-dimensional distributions of $N^{(\Delta)}$ are approximately equal to the finite-dimensional distributions of $N$. In particular, we find, for $n\in\N$,
\begin{align}
\bigg(N\Big(\big(0, \Delta\big] \Big), N\Big(\big(\Delta, 2\Delta\big]\Big),\dots, N\Big(\big((n-1)\Delta, n\Delta\big]\Big)\bigg)& 
\stackrel{\mathrm{d}}{\approx}
\bigg(N^{(\Delta)}\Big(\big(0, \Delta\big] \Big), N^{(\Delta)}\Big(\big(\Delta, 2\Delta\big]\Big),\dots, N^{(\Delta)}\Big(\big((n-1)\Delta, n\Delta\big]\Big) \bigg)\nonumber\\
& = \(X^{( \Delta)}_1,X^{( \Delta)}_2,\dots,X^{( \Delta)}_n\),\label{bin_count_view}
\end{align}
where $\(X^{( \Delta)}_n\)$ is the INAR($\infty$) sequence that the point process $N^{(\Delta)}$ is based on;\ see Theorem~\ref{weak_convergence}. So $\(X^{( \Delta)}_n\)$ is an approximative model for the bin-count sequences of the considered Hawkes process $N$. This point of view can be very fertile. For example, it leads to a nonparametric estimation procedure for the Hawkes process. Instead of fitting a Hawkes process directly, one fits the corresponding INAR($\infty$) model from Theorem~\ref{weak_convergence} on the bin-counts for some small $\Delta>0$. \citet{kirchner15b} gives a detailed discussion of this estimation procedure including asymptotic properties, bias issues and techniques for an optimal bin-size choice $\Delta$. Also note that the Hawkes bin-count sequence view in~\eqref{bin_count_view} on the INAR($\infty$) model is another argument for the choice of the Poisson instead of the Bernoulli distribution for the counting sequences: clearly a Hawkes event can have potentially more than one direct offspring event in a future time-interval.
\par For any $\N_0$-valued time series model, one can construct a point process model the way it is done in Theorem~\ref{weak_convergence}. So, studying integer-valued time series can be inspiring for developing and understanding point process models. For example, one might want to consider the corresponding point process of an integer-valued autoregressive moving-average (INARMA) time series. For INARMA time series, a moving-average part is added to the autoregressive part in the defining difference equations~\eqref{DE}; see \citet{fokianos01}. In fact, the resulting point process is nothing else but the ``dynamic contagion process" as proposed in \citet{zhao12}. Also for integer-valued time series theory, it is inspiring to translate point process models into the discrete-time setup. For example, one might want to translate the generalizing results on selfexciting point processes in~\citet{bremaud96} to the INAR context. Here, the selfexcitement of the point process is modeled not as an affine but as a general Lipschitz function of the past of the process. The analogous generalization in time series theory is ``nonlinear Poisson autoregression'' as studied in \citet{fokianos12} for the case $p=1$. In the latter paper, the authors also find a Lipschitz condition for the transfer function. Yet another idea might be to consider marked INAR sequences in analogy to marked Hawkes processes; see~\citet{liniger09}.
\par One can expect that the results of our paper also hold in a multivariate setup---with the obvious modifications. However, the notation would become even more tedious, so that we have decided to concentrate on the univariate case. In view of the many INAR/Hawkes-correspondences presented in this paper we conclude: INAR($\infty$) sequences are discrete-time versions of Hawkes processes and, vice versa, Hawkes processes are continuous-time versions of INAR($\infty$) sequences. 
\section*{Acknowledgements} 
\noindent  M.K. takes pleasure in thanking Thomas Mikosch for his most valuable comments on an earlier version of the paper and Paul Embrechts for
guiding him to the topic of Hawkes processes and its applications to quantitative risk management. The author acknowledges financial support from RiskLab at the ETH Zurich and the Swiss
Finance Institute. Furthermore, M.K.\ thanks Isabel Marquez da Silva for sharing her expertise on integer-valued models and Rita Kirchner as well as Anne MacKay for help with the editing. 
\section*{References}
\bibliographystyle{elsarticle-harv}
\bibliography{/Users/matthiaskirchner/Desktop/Studium/Diss/Bibliographies/diss.biblio}
\end{document}